\documentclass[10pt]{article}

\usepackage{url}
\usepackage{mathtools}
\usepackage{amssymb}
\usepackage{amsthm}
\usepackage{empheq}
\usepackage{latexsym}
\usepackage{enumitem}
\usepackage{eurosym}
\usepackage{dsfont}
\usepackage{appendix}
\usepackage{color} 
\usepackage[unicode]{hyperref}
\usepackage{frcursive}
\usepackage[utf8]{inputenc}
\usepackage[T1]{fontenc}
\usepackage{geometry}
\usepackage{multirow}
\usepackage[colorinlistoftodos]{todonotes}
\usepackage{lmodern}
\usepackage{anyfontsize}
\usepackage{stmaryrd}
\usepackage{natbib}
\usepackage{cleveref}

\usepackage{amsbsy}

\setcitestyle{numbers,open={[},close={]}}

\definecolor{red}{rgb}{0.7,0.15,0.15}
\definecolor{green}{rgb}{0,0.5,0}
\definecolor{blue}{rgb}{0,0,0.7}
\hypersetup{colorlinks, linkcolor={red},citecolor={green}, urlcolor={blue}}
			
\makeatletter \@addtoreset{equation}{section}

\newtheorem{theorem}{Theorem}[section]
\newtheorem{assumption}[theorem]{Assumption}
\newtheorem{corollary}[theorem]{Corollary}

\newtheorem{proposition}[theorem]{Proposition}

\newtheorem{definition}[theorem]{Definition}
\newtheorem{remark}[theorem]{Remark}

\def\no{\noindent}
\def\beq{\begin{eqnarray}}
\def\eeq{\end{eqnarray}}
\def\be*{\begin{eqnarray*}}
\def\ee*{\end{eqnarray*}}

\def \E{\mathbb{E}}
\def \F{\mathbb{F}}

\def \H{\mathbb{H}}

\def \M{\mathbb{M}}
\def \N{\mathbb{N}}

\def \P{\mathbb{P}}
\def \Q{\mathbb{Q}}
\def \R{\mathbb{R}}
\def \S{\mathbb{S}}

\def \Pr{\mathrm{P}}

\def\Ac{{\cal A}}

\def\Cc{{\cal C}}

\def\Fc{{\cal F}}
\def\Gc{{\cal G}}
\def\Hc{{\cal H}}

\def\Lc{{\cal L}}
\def\Mc{{\cal M}}

\def\Pc{{\cal P}}

\def\Sc{{\cal S}}

\def\Wc{{\cal W}}

\def\x{\times}

\def\Om{\Omega}

\def\om{\omega}

\def\0{\mathbf{0}}

\def \mub{\overline{\mu}}

\def \muh{\widehat{\mu}}

\def \mut{\widetilde{\mu}}

\def \nuh{\widehat{\nu}}

\def\normeL2#1{\left\|{#1}\right\|_{L^2}}

\def\Xh{\widehat X}

\def\esup{{\rm ess \, sup}}

\def \alphab {\boldsymbol{\alpha}}
\def \Ubb {\boldsymbol{U}}

\def \alephbb {\boldsymbol{\aleph}}

\def \Sbb{\mathbf{S}}
\def \Xbb{\mathbf{X}}

\def \Zbb{\mathbf{Z}}

\def \Wbb{\mathbf{W}}

\def \Qr{\mathrm{Q}}

\def \1{\mathds{1}}

\def \alphab {\boldsymbol{\alpha}}
\def \betab {\boldsymbol{\beta}}

\def \Sbb{\mathbf{S}}
\def \Xbb{\mathbf{X}}

\def \Zbb{\mathbf{Z}}

\def \Wbb{\mathbf{W}}

\def\Er{{\rm E}}

\def\Rr{{\rm R}}

\def\Cf{\mathfrak{C}}

\def\If{\mathfrak{I}}

\def \Qr{\mathrm{Q}}

\def \Rr{\mathrm{R}}

\DeclareUnicodeCharacter{014D}{\=o}
 \title{ Three connected problems: principal with multiple agents in cooperation, Principal--Agent with Mckean--Vlasov dynamics and multitask Principal--Agent
 
 }

\author{
 Mao Fabrice Djete\footnote{\'Ecole Polytechnique Paris, Centre de Math\'ematiques Appliqu\'ees, mao-fabrice.djete@polytechnique.edu. This work benefits from the financial support of the Chairs {\it Financial Risk} and {\it Finance and Sustainable Development}} 
    }
             \date{\today}

\begin{document}

\maketitle
 
\begin{abstract}
    In this paper, we address three Principal--Agent problems in a moral hazard context and show that they are connected. We start by studying the problem of Principal with multiple Agents in cooperation. The term cooperation is manifested here by the fact that the agents optimize their criteria through Pareto equilibria. We show that as the number of agents tends to infinity, the principal's value function converges to the value function of a McKean--Vlasov control problem. Using the solution to this McKean--Vlasov control problem, we derive a constructive method for obtaining approximately optimal contracts for the principal's problem with multiple agents in cooperation. In a second step, we show that the problem of Principal with multiple Agents turns out to also converge, when the number of agents goes to infinity, towards a new Principal--Agent problem which is the Principal--Agent problem with Mckean--Vlasov dynamics. This is a Principal--Agent problem where the agent--controlled production follows a Mckean-Vlasov dynamics and the contract can depend of the distribution of the production.  The value function of the principal in this setting is equivalent to that of the same McKean--Vlasov control problem from the multi--agent scenario. Furthermore, we show that an optimal contract can be constructed from the solution to this McKean--Vlasov control problem. We conclude by discussing, in a simple example, the connection of these problems with the multitask Principal--Agent problem which is a situation when a principal delegates multiple tasks that can be correlated to a single agent.
\end{abstract}

\vspace{3mm}
\no{\bf Keywords.} Cooperative equilibrium, multi--agent, moral hazard, Principal--Agent, McKean--Vlasov stochastic differential equations, convergence. 

\vspace{3mm}
\no{\bf MSC2010.} 60K35, 60H30, 91A13, 91A23, 91B30.

\section{Introduction}\label{sec:intro}

The Principal--Agent problem is a fundamental issue in economics and organizational theory that arises when one party (the principal) delegates work to another party (the agent), who is responsible for performing a task on the principal's behalf. The key challenge is that the agent may have different goals, preferences, or information than the principal, creating a potential for conflict. Because the principal cannot perfectly observe or control the agent's actions (moral hazard), the agent may act in their own self--interest, which may not align with the best interests of the principal. 

\medskip
This problem is especially relevant in situations involving contracts, incentives, and performance measurement, and it highlights the importance of designing mechanisms that align the agent’s behavior with the principal’s objectives.
From a game theory standpoint, the goal of aligning the interests of the principal and the agent translates into the search for a Stackelberg equilibrium. In simple terms, the principal, when offering a contract, anticipates the agent's optimal response. Based on this anticipated best response, the principal then seeks to maximize her own utility and identify the optimal contract.

\medskip
The Principal--Agent problem has been widely explored in economics, with an extensive literature addressing its mathematical formulation. Notable references, such as \citeauthor*{LaffontMartimort2002} \cite{LaffontMartimort2002}, \citeauthor*{SchmitzReview} \cite{SchmitzReview}, \citeauthor*{laffont1993theory} \cite{laffont1993theory}, and \citeauthor*{salanie1997economics} \cite{salanie1997economics}, offer a broad overview of the subject while tackling various questions it raises. Initially, the mathematical resolution of the Principal--Agent problem was confined to static or discrete--time settings. This changed with the seminal work of \citeauthor*{Milgrom1987} \cite{Milgrom1987}, who introduced a continuous--time approach. Building on the work did in \cite{Milgrom1987}, \citeauthor*{schattler1993first} \cite{schattler1993first,schattler1997optimal}, \citeauthor*{sung1995linearity} \cite{sung1995linearity}, and \citeauthor*{muller1998first} \cite{muller1998first,muller2000asymptotic} expanded the analysis to various scenarios.

\medskip
The use of stochastic control theory tools has advanced significantly thanks to the contributions of \citeauthor*{sannikov2008continuous} \cite{sannikov2008continuous, sannikov2012contracts}, who introduced innovative tools for tackling the continuous--time Principal--Agent problem. Inspired by this, \citeauthor*{cvitanic2015dynamic} \cite{cvitanic2015dynamic, cvitanic2014moral} developed a comprehensive framework for addressing the problem. Their approach can be summarized as follows: the agent's problem is solved using a Backward Stochastic Differential Equation (BSDE) based on the dynamic programming principle, while the principal's problem is framed as a stochastic control problem with two states  variables: the output controlled by the agent and the agent's continuation utility.

\medskip
An interesting extension of the traditional Principal--Agent framework is the problem of Principal with multiple Agents.  The problem of the principal with multiple agents is a contract theory problem, where a single principal engages with multiple agents, each tasked with performing separate or interdependent activities on the principal’s behalf. This setting introduces new complexities, primarily due to inter--agent relationships and coordination issues. This problem is sometimes referred to as the multi--Agent Principal--Agent problem. 
The key challenges arise from the potential for competition or cooperation among agents, free--riding, and the difficulty of allocating rewards when agents' efforts are interdependent. The principal must design contracts or incentive mechanisms that not only align the interests of each individual agent with their own objectives but also promote cooperation or healthy competition among agents when necessary.

\medskip
This model is commonly applied in corporate management, where an employer (Principal) must incentivize different departments or employees (Agents) to work together effectively. Solutions often focus on crafting contracts that consider joint output, relative performance, or team--based rewards, all while mitigating risks like shirking, miscommunication, or opportunistic behavior. 
The multi--Agent Principal--Agent problem has been addressed using various approaches and techniques in numerous papers and books, including \citeauthor*{holmstrom1982moral} \cite{holmstrom1982moral}, \citeauthor*{Mookherjee1984} \cite{Mookherjee1984}, \citeauthor*{green1983comparison} \cite{green1983comparison}, \citeauthor*{demski1984optimal} \cite{demski1984optimal}, \citeauthor*{SungOptimal2008} \cite{SungOptimal2008}, \citeauthor*{possamai2019contracting} \cite{possamai2019contracting}, $\cdots$

\medskip
In this paper, we consider a moral hazard setting. We are investigating the problem of Principal with multiple Agents in cooperation. Here, the term cooperation refers to the fact that the agents are optimizing their criterion through a Pareto equilibrium. We start by characterizing the problem of the principal with $n$ agents when $n$ goes to infinity. This characterization is done by first showing that the problem of the principal converges to a specific limit problem when $n$ goes to infinity. Second, by using the limit problem obtained, we can construct a sequence of contracts that are approximately optimal for the problem of Principal with $n$ Agents when $n$ is sufficiently large enough. The limit problem turns out to be a McKean--Vlasov optimal control problem. This convergence result is obtained by beginning to provide a reformulation of the problem of the principal with $n$ agents similarly to the one used in \cite{cvitanic2015dynamic} through dynamic programming and BSDEs. Then, with the help of the techniques developed in \citeauthor*{djete2019general} \cite{djete2019general}, we are able to show the convergence of the value function of the principal and construct approximately optimal contracts for the problem of Principal with $n$ Agents from an optimal control of the limit problem. This situation of cooperation creates new difficulties compared to the case of competition studied  with Nash equilibrium in \cite{possamai2019contracting},  \citeauthor*{mastroliaAtale2019} \cite{mastroliaAtale2019}, \citeauthor*{HubertMean2021} \cite{HubertMean2021}, \citeauthor*{DayanikliOptimal2022} \cite{DayanikliOptimal2022}, \citeauthor*{Carmona2018FiniteStateCT} \cite{Carmona2018FiniteStateCT}, \citeauthor*{MFD_PA_comp} \cite{MFD_PA_comp}, \citeauthor*{BergaultPA_minor_major_2024} \cite{BergaultPA_minor_major_2024}, $\cdots$ Indeed, one main advantage of the Nash equilibrium situation is the fact that, given a proposed contract to the $n$ agents,  a deviating agent/player has a negligible affect on this contract when $n$ goes to infinity. This is no longer the case in cooperation since we are no longer considering one deviating player. All the players can deviate.

\medskip
The problem of Principal with multiple Agents in cooperation has another natural connection which is the Principal--Agent problem with McKean--Vlasov dynamics. What we mean by Principal--agent problem with McKean--Vlasov dynamics is a Principal--Agent problem where the dynamics of the production, which is controlled by the agent, is of McKean--Vlasov type. Meaning, the distribution of the production and/or control appear in the coefficients. Furthermore, the contract proposed by the Principal does no longer depend just on a the production (moral harzard) but the distribution of the production can also be part of the contract offered by the Principal. To the best of our knowledge this type of Principal--Agent problem has never been investigated in the literature. We are able to show that the Principal--Agent problem with McKean--Vlasov dynamics is the limit of the problem of Principal with $n$ Agents when $n$ goes to infinity. {\color{black}Namely, any convergent sequence of contracts solving the problem of Principal with $n$ agents has a limit connected to a contract solving the Principal--Agent problem with McKean--Vlasov dynamics. Also, from any contract solving the Principal--Agent problem with McKean--Vlasov dynamics, we can construct a sequence of contracts that solve approximately the problem of Principal with $n$ agents when $n$ is large.}

\medskip
In addition, we study the Principal--Agent problem with McKean--Vlasov dynamics by giving a characterization of the optimal control of the agent, by providing the shape of the optimal contracts offered to the agent and by showing that the value function of the principal is equal to the value function of a McKean--Vlasov control problem similar to the one involved in the problem of Principal with multiple Agents in cooperation. This study raises some significant issues since the classical approach by dynamics programming and BSDEs developed in \cite{cvitanic2015dynamic} does not work anymore because of the distribution dependence. We overcome the difficulty with the help of an $n$--player approximation of the production/value function combining with the use of the weak limit of convergent sequence of BSDEs.

\medskip
Lastly, in a simple situation, we explore the Multitask Principal--Agent problem. It has been introduced in \citeauthor*{Holmstrom_Milgrom_multitask} \cite{Holmstrom_Milgrom_multitask}. The Multitask Principal--Agent problem emerges when a principal delegates multiple tasks to a single agent, each with potentially different characteristics in terms of difficulty, measurability, or importance. The challenge for the principal is to design incentives that encourage the agent to allocate effort across all tasks efficiently, rather than focusing only on tasks that are easier to measure or reward.

\medskip
An issue in this model can be the measurement problem: if some tasks are easier to observe or quantify, the agent may prioritize these tasks while neglecting others that are harder to measure but equally or more important. To address this, the principal must balance incentives, often opting for lower--powered incentives to prevent overemphasis on certain tasks. This problem is relevant in various fields: corporate management (where employees must balance short-term performance with long-term goals), education (where professors balance teaching, research, and service), contractual situations between technology companies like Google, Meta and content creator, $\cdots$. We revisit a simple example formulated in \cite{Holmstrom_Milgrom_multitask} for homogeneous tasks. While the interpretation is different, we will see that this problem is equivalent to the problem of Principal with multiple Agents in cooperation and solve it, when $n$ goes to infinity, by using the McKean--Vlasov control problem.

\medskip
The paper is structured as follows. After briefly reviewing some notations, \Cref{sec:main} outlines the framework considered and presents the main results. Specifically, \Cref{ssec:equidyn} formulates the problem of Principal with $n$ Agents and the corresponding limit problem i.e. the McKean--Vlasov control problem while stating the convergence results. In \Cref{section:PA_McKean}, the Principal--Agent problem with McKean--Vlasov dynamics is introduced and results are stated about its connection with the problem of Principal with $n$ Agents. Finally, \Cref{sec:multitask} is devoted to the discussion and resolution of a simple multitask Principal--Agent problem. All proofs are provided in \Cref{sec:proofs}.

\medskip
{\bf \large Notations}.
	$(i)$
	Given a {\color{black}Polish} space $(E,\Delta)$ and $p \ge 1,$ we denote by $\Pc(E)$ the collection of all Borel probability measures on $E$,
	and by $\Pc_p(E)$ the subset of Borel probability measures $\mu$ 
	such that $\int_E \Delta(e, e_0)^p  \mu(de) < \infty$ for some $e_0 \in E$. The topology induced by the weak convergence will be used for $\Pc(E)$ and, for $p \ge 1$, we equip $\Pc_p(E)$ with the Wasserstein metric $\Wc_p$ defined by
	\[
		\Wc_p(\mu , \mu') 
		~:=~
		\bigg(
			\inf_{\lambda \in \Lambda(\mu, \mu')}  \int_{E \x E} \Delta(e, e')^p ~\lambda( \mathrm{d}e, \mathrm{d}e') 
		\bigg)^{1/p},
	\]
	where $\Lambda(\mu, \mu')$ denotes the collection of all probability measures $\lambda$ on $E \x E$ 
	such that $\lambda( \mathrm{d}e, E) = \mu$ and $\lambda(E,  \mathrm{d}e') = \mu'( \mathrm{d}e')$. Equipped with $\Wc_p,$ $\Pc_p(E)$ is a Polish space (see \cite[Theorem 6.18]{villani2008optimal}). 

\medskip
    \noindent $(ii)$ 
	Given a probability space $(\Om, \Hc, \P)$ supporting a sub--$\sigma$--algebra $\Gc \subset \Hc$ then for a Polish space $E$ and any random variable $\xi: \Om \longrightarrow E$, both the notations $\Lc^{\P}( \xi | \Gc)(\om)$ and $\P^{\Gc}_{\om} \circ (\xi)^{-1}$ are used to denote the conditional distribution of $\xi$ knowing $\Gc$ under $\P$.

\medskip
	\noindent $(iii)$	
	For a Polish space $E$, 
    we denote by $\M(E)$ the space of all Borel measures $q( \mathrm{d}t,  \mathrm{d}e)$ on $[0,T] \x E$, 
	whose marginal distribution on $[0,T]$ is the Lebesgue measure $ \mathrm{d}t$, 
	that is to say $q( \mathrm{d}t, \mathrm{d}e)=q(t,  \mathrm{d}e) \mathrm{d}t$ for a family $(q(t,  \mathrm{d}e))_{t \in [0,T]}$ of Borel probability measures on $E$.
	For any $q \in \M(E)$ and $t \in [0,T],$ we define $q_{t \wedge \cdot} \in \M(E)$ by $q_{t \wedge \cdot}(\mathrm{d}s, \mathrm{d}e) :=  q(\mathrm{d}s, \mathrm{d}e) \big|_{ [0,t] \x E} + \delta_{e_0}(\mathrm{d}e) \mathrm{d}s \big|_{(t,T] \x E},\; \text{for some fixed $e_0 \in E$.}$ 
    In the case where $E$ is a normed space of norm $|\cdot|_E$, we will use the notation $|\cdot|$ for the norm regardless of the space $E$ when the space is obvious. For a measure $\lambda$ on $E$, we will refer to $\|\lambda\|_p$ for $\left(\int_E |e|^p \lambda(\mathrm{d}e) \right)^{1/p}.$

\medskip
    \noindent $(iv)$
	Let $\N^*$ denote the set of positive integers. Let $T > 0$ and $(\Sigma,\rho)$ be a Polish space, we denote by $C([0,T]; \Sigma)$ the space of all continuous functions on $[0,T]$ taking values in $\Sigma$.
	When $\Sigma=\R^k$ for some $k\in\N^*$, we simply write $\Cc^k := C([0,T]; \R^k)$. 

\medskip
	With $k \in \N^*$ and a Polish space $E$,
	a map $h:[0,T] \x \R^k \x C([0,T];\Sigma) \x \M(E)$ is called progressively Borel measurable if it verifies $h(t,x,\pi,q)=h(t,x,\pi_{t \wedge \cdot},q_{t \wedge \cdot}),$ for any $(t,x,\pi,q) \in [0,T]\x \R^k \x C([0,T];\Sigma) \x \M(E).$

\section{Problem formulation and main results} \label{sec:main}


The general assumptions used throughout this paper are now formulated. The nonempty Polish spaces $A \subset \R$ and $\Er \subset \R$, $1 \le  p < 2$ and the {\color{black}time horizon} $T>0$ are fixed. 
We set the probability space $(\Om,\H:=(\Hc_t)_{t \in [0,T]},\Hc,\P)$\footnote{ \label{footnote_enlarge}The probability space $(\Om,\H,\P)$ contains as many random variables as we want in the sense that: each time we need a sequence of independent uniform random variables or Brownian motions, we can find them on $\Om$ with the distribution $\P$ without mentioning an enlarging of the space. }, and $\nu \in \Pc(\R)$ having all its exponential moments finite i.e. $\int_\R \exp(a|u|)\nu(\mathrm{d}u)< \infty$ for any $a \ge 0$. We are given the following progressively Borel measurable functions:
	\[
		\left(b, L \right):[0,T] \x \R \x C([0,T];\Pc_p(\R)) \x \Er \x A \longrightarrow \R \x \R,\;\;{L}_{\Pr}:[0,T] \x \Er \longrightarrow \R\;\;\mbox{and}\;\;\sigma:[0,T] \x \R \to \R,
	\]
 and the Borel maps  $U:\R \to \R$, $\Upsilon:\R \to \R$, and $({g},{g}_{\Pr}):  C([0,T];\Pc_p(\R)) \x \Er \to \R \x \R$.

    \begin{assumption} \label{assum:main1} 
		
		
		
		\medskip
		$(i)$ The map $[0,T] \x \R  \x \Pc_p(\R) \x \Er \x A \ni (t,x,\pi,e,a) \mapsto \left( b(t,x,\pi,e,a), \sigma(t,x) \right) \in \R \x \R $ is Lipschitz in $(x,\pi)$ uniformly in $(t,e,a)$, continuous in $(x,\pi,e,a)$ for any $t$, and with linear growth in $(a,e)$ uniformly in $(t,x,\pi)$;

    \medskip
		$(ii)$ {\rm Non--degeneracy condition and boundedness:} 
        $0 < \inf_{(t,x)}  \sigma(t,x)^2 \le \sup_{(t,x)} \sigma(t,x)^2 < \infty$; 
		
    \medskip
		$(iii)$ 
        The map $ (t,x,\pi,e,a) \mapsto L(t,x,\pi,e,a) $  is continuous in $(x,\pi,e,a)$ for each $t$, and there exist $\overline{c}, \widehat{c} > 0$ s.t. for all $(t,x,\pi,e,a)$,
        \begin{align*}
            -\overline{c}\left( 1 + |x|^2 + \sup_{s \le t}\|\pi(s)\|_2^2 + |e|^2 + |a|^2 \right) \le L(t,x,\pi,e,a) \le \overline{c}\left( 1 + |x|^p + \sup_{s \le t}{\color{black}\| \pi(s)\|_p^p} \right)- \widehat{c} \left(|e|^2 + |a|^2 \right)
        \end{align*}
        Also, the map $U$ has linear growth and the map $ (t,\pi,e,v,u) \mapsto \left( {L}_{\Pr}(t,e), (g,g_{\Pr})(\pi,e),\Upsilon(v),U(u) \right)$  is continuous in $(\pi,e,v,u)$ for each $t$.

		
	\end{assumption}


\subsection{The problem of the principal with $n$--player in cooperation}\label{ssec:equidyn}
In this section, we give the mathematical framework we will consider for the problem of the principal with $n$ agents in cooperation. We start by providing what we will call admissible contracts and admissible controls. On the filtered probability space $(\Om,\H,\Hc,\P),$ let $(W^i)_{i \in \N^*}$ be a sequence of independent $\R$--valued $\H$--adapted  Brownian motions, and $(\iota^i)_{i \in \N^*}$ a sequence of iid $\Hc_0$--random variables of law $\nu.$ Besides, $(W^i)_{i \in \N^*}$ and $(\iota^i)_{i \in \N^*}$ are independent. 

\paragraph*{Admissible contracts}
We will say $\Cf^n=(\xi^n,\alephbb^n=(\aleph^{1,n},\cdots,\aleph^{n,n}))$ is a contract if $\xi^n:\Cc^n \to \Er$ is Borel measurable, $\aleph^{i,n}:[0,T] \x \Cc^n \to \Er$ is progressively Borel measurable and
\begin{align*}
    \E \left[ \exp{ \left(a |g \left(\nu^n,\xi^n(\Ubb^n)\right)| + \sum_{i=1}^n\int_0^T b|\aleph^{i,n}(t,\Ubb^n)|^2 \mathrm{d}t \right) } \right]< \infty,\;\mbox{for any }a,b \ge 0,
\end{align*}
where $\Ubb^n=(U^1,\cdots,U^n)$ with $U^i_\cdot=\iota^i + \int_0^\cdot\sigma(t,U^i_t)\mathrm{d}W^i_t$ and $\nu^n_t:=\frac{1}{n} \sum_{i=1}^n \delta_{U^i_t}$. We denote by $\Xi_n$ the set of all contracts.

\paragraph*{\color{black}Admissible controls}
We denote by $\Ac_n$ the set of progressively Borel maps $\beta:[0,T] \x \Cc^n \to A$. We say $\betab=(\beta^1,\cdots,\beta^n) \in \Ac_n^n$ is admissible if:
\begin{align*}
    \E \left[ \exp{ \left(\sum_{i=1}^n\int_0^T b|\alpha^{i,n}(t,\Ubb^n)|^2 \mathrm{d}t \right) } \right]< \infty,\;\mbox{for any }b \ge 0.
\end{align*}

\begin{remark}
    $(i)$ In a contract $\Cf^n=(\xi^n,\alephbb^n)$ proposed by the principal, the map $\xi^n$ refers to the final payment received by the $n$ agents and the maps $\alephbb^n=(\aleph^{1,n},\cdots,\aleph^{n,n})$ are the instantaneous payments given to the $n$ agents at time $t \in [0,T]$.

    $(ii)$ It is worth mentioning that the integrability conditions over the contracts and the controls can be weaken $($see for instance {\rm \Cref{section:PA_McKean}} for the Principal--Agent problem with McKean--Vlasov dynamics$)$. We chose these conditions to be consistent with the literature and because they do not have a major impact on most of our results.
\end{remark}

\subsubsection{The $n$--agent problem}
Let $n \ge 1$. We start by providing here the problem faced by the $n$ agents. Given an admissible contract ${\Cf}^n:=\left(\xi^n,\alephbb^n=(\aleph^{i,n})_{1 \le i \le n} \right)$ and the admissible controls $\alphab^n:=(\alpha^{1,n},\dots,\alpha^{n,n}) \in \Ac_n^n$, we denote by $\Xbb^{\alephbb^n,\alphab^n}:=\Xbb=(X^1_{\cdot},\dots,X^n_{\cdot})$ the process satisfying: for each $1 \le i \le n,$ $X^i_0:=\iota^i$, and $\P$--a.e.
    \begin{align} \label{eq:N-agents_StrongMV_CommonNoise-law-of-controls}
        \mathrm{d}X^i_t
        =
        b\left(t,X^i_{t}, \mu^n ,\aleph^{i,n}(t,\Xbb),\alpha^{i,n}(t,\Xbb) \right) \;\mathrm{d}t 
        +
        \sigma(t,X^i_t) \mathrm{d}W^i_t\;\;\mbox{with}\;\mu^n=(\mu^n_t)_{t \in [0,T]}\;\mbox{and}\;\mu^n_t := \frac{1}{n}\sum_{i=1}^n \delta_{X^{i}_{t}}.
    \end{align}
The process $\Xbb$ has to be seen as the production of the agents which is managed by using the controls $\alphab^n$. The reward value associated with the contract $\Cf^n=(\xi^n,\alephbb^n)$ and the control rule/strategy $\alphab^n:=(\alpha^{1,n},\dots,\alpha^{n,n})$ is then defined by
	\begin{align*}
	    J^{\Cf^n}_{n}(\alphab^n)
        :=
        \E\left[ \Rr^{\Cf^n}_n(\alphab^n) \right],\;
	    \Rr^{\Cf^n}_n(\alphab^n):=\frac{1}{n} \sum_{i=1}^n
        \int_0^T 
        L\left(t,X^{i}_t,\mu^n,\aleph^{i,n}(t,\Xbb),\alpha^{i,n}(t,\Xbb) \right)\;\mathrm{d}t 
        +
        g\left( \mu^n,\xi^n(\Xbb) \right)
        .
	\end{align*}

\begin{remark}
    It is worth mentioning that {\rm\Cref{eq:N-agents_StrongMV_CommonNoise-law-of-controls}} is not necessarily well--posed in the strong sense. This is essentially due to the assumptions over the coefficients and the fact that $\alephbb^n$ and $\alphab^n$ are just Borel measurable. However,  by using Girsanov's Theorem, we can see that {\rm\Cref{eq:N-agents_StrongMV_CommonNoise-law-of-controls}} is uniquely defined in distribution. The process $\Xbb$ should thus be regarded as a weak solution. While it may be necessary to extend  $(\Om,\H,\P)$ to express $\Xbb$, for simplicity and to avoid cumbersome notation, we have assumed that $(\Om,\H,\P)$  is already suitable, requiring no further extensions $($see {\rm \Cref{footnote_enlarge}}$)$.
\end{remark}

We are giving now the notion of equilibrium used here. The meaning of agents in cooperation is manifested by the use of this notion of equilibrium.
\begin{definition}{\rm (Pareto Equilibrium)} \label{def:Nplayers--equilibria}

\medskip    
    Let $n \ge 1$  and the contract $\Cf^n:=(\xi^n,\alephbb^n)$. We will say that an admissible control rule/strategy $(\alpha^{1,n},\dots,\alpha^{n,n}) \in \Ac_n^n$ is an {\rm equilibrium} if 
	    \[
	        J^{\Cf^n}_{n}(\alpha^{1,n},\dots,\alpha^{n,n}) \ge \sup_{(\beta^1,\cdots,\beta^n) \in \Ac_n^n} J^{\Cf^n}_{n}\big( \beta^1,\cdots,\beta^n \big).
	    \]    
\end{definition}
We set $\Ac_n^{\star,n}(\Cf^n)
    :=
    \left\{ \alphab^n=(\alpha^{1,n},\cdots,\alpha^{n,n})\mbox{ an equilibrium given }\Cf^n=(\xi^n,\alephbb^n) \right\}.$

\subsubsection{The Principal problem}

For a contract $\Cf^n$, when $\Ac_n^{\star,n}(\Cf^n)$ is non-empty, it can contained multiple solutions. In such cases, the principal is responsible for selecting the agents' controls, as commonly done in the literature. The interpretation is twofold: first, the principal must ensure that the agents will accept the proposed contract and anticipate their likely responses. Second, once the contract is offered, the principal also suggests a specific course of action to the agents. However, there exists a minimum level of utility, known as the reservation utility and denoted $R \in \R$ below which the agents will reject the contract. With this mechanism in mind, we then proceed to formulate the Principal's problem.

\medskip
Let us fixed $R \in \R$. Given a contract $\Cf^n$, an optimal control $\alphab^n \in \Ac_n^{\star,n}(\Cf^n)$ satisfies the reservation utility if $\E\left[ \Rr^{\Cf^n}_n(\alphab^n) \big| \Xbb_0 \right] \ge R$ a.e. We denote by ${\rm PE}[\Cf^n]$ all this type of admissible controls. The problem faced by the Principal is
\begin{align*}
    V_{n,\Pr}
    :=
    \sup_{\Cf^n \in \Xi_n} \;\;\sup_{\alphab^n \in {\rm PE}[\Cf^n]} J_{n,\Pr}^{{\color{black}\alphab^n}}(\Cf^n)
\end{align*}
where the reward of the Principal is
\begin{align*}
    J_{n,\Pr}^{{\color{black}\alphab^n}}(\Cf^n)
    :=
    &\E \left[ U \left( \frac{1}{n} \sum_{i=1}^n \Upsilon \left(X^i_T\right)
    -g_{\Pr} \left(\mu^n,\xi^n(\Xbb) \right)
    -
    \int_0^T L_{\Pr} \left(t,\aleph^{i,n}(t,\Xbb) \right)\;\mathrm{d}t\right) \right].
\end{align*}

\begin{remark}
    If we follow the literature on the Principal--Agent problem, the usual condition we must consider for the minimum level of utility of the agents is $J^{\Cf^n}_{n}(\alphab^n) \ge R$. However, this condition is not appropriate here since the initial production $\Xbb_0$ is not deterministic as in the classical literature. The randomness of $\Xbb_0$ needs to be taken into consideration, hence the use of the condition we mentioned.
\end{remark}

\subsubsection{Control McKean--Vlasov formulation}
 \label{ssec:MFGformulation}

As we will see, the problem of the principal with $n$ agents is related to a McKean--Vlasov control problem when $n$ goes to infinity. Let us descrite here the limit problem in question. 
We introduce the map $\widehat{\alpha}:[0,T] \x \R \x C([0,T];\Pc_p(\R)) \x \Er \x \R \to A$ by: 
for any $(t,x,\pi,e,z)$, $\widehat{\alpha}(t,x,\pi,e,z)$ is the $\underline{ {\rm unique\;maximizer}}$ of the Hamiltonian 
\begin{align} \label{eq:maximizee}
    A \ni a \mapsto  b \left(t,x,\pi,e,a \right) z +  L \left(t,x,\pi,e,a \right) \in \R.
\end{align}
For each  $(t,x,\pi,e,z),$ we set $(\widehat{b}, \widehat{L} ) \left( t, x,\pi, e , z \right):= \left( b, L \right) \left(t,x, \pi,e, \widehat{\alpha} \left( t, x,\pi, e , \sigma(t,x)^{-1} z\right) \right)$
and
\begin{align*}
    H\left( t, x,\pi, e , z\right)
    :=
    \widehat{b}\left( t, x,\pi, e , z\right) \sigma(t,x)^{-1}z + \widehat{L}\left( t, x,\pi, e , z \right).
\end{align*}
We also define $h(t,x,\pi,e,z,a):=b \left(t,x,\pi,e,a \right) \sigma^{-1}(t,x)z +  L \left(t,x,\pi,e,a \right)$.
\begin{assumption} \label{assump:growth_coercivity}
    \medskip
		The map $e \mapsto g(\pi,e)$ is invertible for any $\pi$. 
        The map $(t,x,\pi,e,z) \mapsto (\widehat{b}, \widehat{L} )(t,x,\pi,e,z)$ is continuous in $(x,\pi,e,z)$ for each $t$, the map $(t,x,\pi,e,z) \mapsto \widehat{b}(t,x,\pi,e,z)$ is Lipschitz in $(x,\pi)$ uniformly in $(t,e,z)$, and with the positive constants $\overline{c},$ for all $(t,x,\pi,e,z)$
    \begin{align*}
       |\widehat{\alpha}(t,x,\pi,e,z)| \le \overline{c}\left(1 + |x| + \sup_{s \le t}\|\pi(s)\|_p + |e| + |z| \right).
    \end{align*}

\end{assumption}
We introduce $\widehat{g}_{\Pr}(\pi,y):=g_{\Pr}(\pi,g^{-1}(\pi,y))$ for all $(\pi,y)$ where $g^{-1}(\pi,\cdot)$ is the inverse of $y \mapsto g(\pi,y)$ for each $\pi$.

\begin{assumption} \label{assump:growth_coercivity1}
   { \rm (Coercivity)}
    The map $U$ is non--decreasing, concave and satisfies $\lim_{k \to -\infty} U(k)=-\infty$. With the positive constants $\overline{c}, \widehat{c}$, for all $(t,x,\pi,e,z),$ 
\begin{align*}
     \overline{c} \left( 1 + |x|^{p} +  \sup_{s \le t}\|\pi(s)\|^p_p  \right) - \widehat{c} \left(|e|^{2} + |z|^{2} \right) \ge \widehat{L} \left( t, x,\pi, e,  z \right) \ge  -\overline{c}\left(  1 + |x|^{2} + \sup_{s \le t}\| \pi(s) \|_{2}^{2} + |e|^{2} + |z|^{2} \right)
\end{align*}
and $|\Upsilon(x)| \le \overline{c} \left( 1 + |x|^p \right),\;\; \overline{c}(1+ |e|^{2}) \ge L_{\Pr} \left( t, e \right) \ge \widehat{c} |e|^{2} - \overline{c}$,
\begin{align*}
    \overline{c}\left(1+ \sup_{t \le T}\|\pi(t)\|_p^p + |y|  \right) \ge \widehat{g}_{\Pr}(\pi,y) \ge \widehat{c} y-\overline{c}\left(1+ \sup_{t \le T} \|\pi(t)\|_p^p \right).
\end{align*}
\end{assumption}

\medskip
 Under these assumptions, we can provide the limit problem of $V_{n,\Pr}$ when $n \to \infty$.
 We will say that $(\gamma, \aleph)$ is admissible and belongs to $\Ac$ if the map $(\gamma,\aleph):[0,T] \x \R \to \R \x \Er$ is Borel and, the process $X^{\gamma,\aleph}:=X$ is well defined i.e. $X_0 = \iota$, 
 $\E \left[\int_0^T |\aleph(t,X_t)|^2 + |\gamma(t,X_t)|^2 \mathrm{d}t \right] < \infty$,
\begin{align*}
   \mathrm{d}X_t =  \widehat{b}\left(t,X_{t},\mu,\aleph(t,X_t),\gamma(t,X_t)  \right) \;\mathrm{d}t 
        +
        \sigma(t,X_t) \mathrm{d}W_t\;\;\mbox{with}\;\;\mu=(\mu_t)_{t \in [0,T]}\;\mbox{and}\;\mu_t=\Lc(X_t).
\end{align*}
With the process $X^{\gamma,\aleph}$ and a constant $Y_0$ s.t. $Y_0 \ge R$, we define $Y^{\gamma,\aleph}$ by
\begin{align*}
    Y^{\gamma,\aleph}_\cdot
    :=
    Y_0 - \E \left[\int_0^\cdot \widehat{L}\left(t,X^{\gamma,\aleph}_t,\mu,\aleph(t,X^{\gamma,\aleph}_t),\gamma(t,X^{\gamma,\aleph}_t) \right)\;\mathrm{d}t \right].
\end{align*}
We can now introduce the value function given by 
\begin{align} \label{eq:limit_value}
    \widehat{V}
    :=
    \sup_{(\gamma,\aleph) \in \Ac} \widehat{J}(\gamma,\aleph)\;\mbox{where}\;\widehat{J}(\gamma,\aleph)
    :=
    \E \left[  \Upsilon \left(X^{\gamma,\aleph}_T\right)
    -\widehat{g}_{\Pr} \left(\mu,Y^{\gamma,\aleph}_T \right)
    -
    \int_0^T L_{\Pr} \left(t,\aleph(t,X^{\gamma,\aleph}) \right)\;\mathrm{d}t 
    \right].
\end{align}
\begin{theorem} \label{thm:cong_PA_comp}
    Under {\rm \Cref{assum:main1}}, {\rm \Cref{assump:growth_coercivity}} and {\rm \Cref{assump:growth_coercivity1}}, $\lim_{n \to \infty} V_{n,\Pr} =U (\widehat{V})$.
\end{theorem}

\begin{remark}
    $(i)$ To the best of our knowledge, the convergence of the problem of the principal with $n$ agents in cooperation when $n$ goes to infinity has never been formulated or study in the literature. {\rm \Cref{thm:cong_PA_comp}} gives an answer to this question. Some assumptions may be weaken such as the uniqueness of the maximizer of the Hamiltonian $h$. But in order to keep the formulation and statement clear and simple, we have chosen to stick with this assumption.

    \medskip
    $(ii)$ With additional assumptions on the coefficients, it is indeed possible to give a rate for this convergence result as in {\rm \Cref{prop:multitask}} $($see below$)$. As we want to avoid a too heavy presentation, we have not presented this type of result
\end{remark}

Let us now show how we can use the solution of the McKean--Vlasov control problem to approximately solve the problem of Principal with $n$ Agents when $n$ is large. Let $(\gamma,\aleph) \in \Ac$ be an optimal solution for $\widehat{V}$ s.t.
\begin{align} \label{eq:opti_control_cond}
    \E \left[\exp \left(\int_0^T a |\aleph(t,U_t)|^2 + b |\gamma(t,U_t)|^2 \mathrm{d}t \right) \right]< \infty,\;\;\mbox{for each }a,b \ge 0
\end{align}
with $U_0=\iota$ and $\mathrm{d}U_t=\sigma(t,U_t)\mathrm{d}W_t$. Let $\gamma^\ell(\cdot):=\gamma(\cdot) \wedge \ell$, $\aleph^\ell(\cdot):=\aleph (\cdot) \wedge \ell,$ $\Xbb^{\ell,n}:=\left(X^{1},\cdots,X^{n} \right)$ be the solution of 
\begin{align*}
    \mathrm{d}X^{i}_t =  \widehat{b}\left(t,X^{i}_{t}, \mu^{\ell,n},\aleph^\ell(t,X^{i}_t),\gamma^\ell(t,X^{i}_t)  \right) \;\mathrm{d}t 
    +
    \sigma(t,X^{i}_t) \mathrm{d}W^i_t
\end{align*}
with $\mu^{\ell,n}_t=\frac{1}{n} \sum_{i=1}^n \delta_{X^{i}_t}$.
We set
\begin{align*}
    &Y^{\ell,n}_\cdot
    :=
    Y^{\gamma,\aleph}_0 -\frac{1}{n} \sum_{i=1}^n\int_0^\cdot H\left(t,X^{i}_t,\mu^{\ell,n},\aleph^\ell(t,X^{i}_t),\gamma^\ell(t,X^{i}_t) \right)\;\mathrm{d}t + \frac{1}{n} \sum_{i=1}^n \sigma(t,X^{i}_t)^{-1}\gamma(t,X^{i}_t) \mathrm{d}X^{i}_t,
    \\
    &\;\;\xi^{\ell,n}(\Xbb^{\ell,n}):=g^{-1}\left(\mu^{\ell,n},Y^{\ell,n}_T \right)\;\;\mbox{and}\;\;\aleph^{\ell,i}(t,\Xbb^{\ell,n}):=\aleph^\ell(t,X^{i}_t).
\end{align*}
We consider for each $\ell, n \ge 1$, the contract $\Cf^{\ell,n}:=\left(\xi^{\ell,n},(\alephbb^{\ell,n}:=\aleph^{\ell,i})_{1 \le i \le n} \right)$ and we also define the control $\alphab^{\ell,n}:=\left( \alpha^{\ell,1},\cdots, \alpha^{\ell,n} \right)$ by
\begin{align*}
        \alpha^{\ell,i}\left( t, \Xbb^{\ell,n}\right)
        :=
        \widehat{\alpha}\left(t,X^{i}_{t}, \mu^{\ell,n},\aleph^{\ell,i}(t,\Xbb^{\ell,n}),\sigma(t,X^{i}_t)^{-1}\gamma(t,X^{i}_t)  \right).
    \end{align*}
\begin{theorem} \label{thm:fromLimit}
    Under {\rm \Cref{assum:main1}},  {\rm \Cref{assump:growth_coercivity}} and {\rm \Cref{assump:growth_coercivity1}}, the sequence of contracts $(\Cf^{\ell,n})_{\ell, n \ge 1}$ and the sequence of controls $(\alphab^{\ell,n})_{\ell,n \ge 1}$ are s.t. 
    \begin{align*}
        \lim_{\ell \to \infty}\lim_{n \to \infty} \left| J_{n,\Pr}^{{\color{black}\alphab^{n}}}(\Cf^{\ell,n}) - V_{n,\Pr} \right|=0
    \end{align*}
    and for each $\ell, n \ge 1$, $\alphab^{\ell,n} \in \mbox{\rm PE}[\Cf^{\ell,n}]$ i.e. $\alphab^{\ell,n}$ is admissible,  $\E\left[ \Rr^{\Cf^n}_n(\alphab^n) \big| \Xbb_0 \right] \ge R$ a.e. and
    \begin{align*}
        J^{\Cf^{\ell,n}}_{n}(\alphab^{\ell,n}) = \sup_{\beta^1,\cdots,\beta^n}J^{\Cf^{\ell,n}}_{n}(\beta^1,\cdots,\beta^n).
    \end{align*}
\end{theorem}

\begin{remark}
    $(i)$ {\rm \Cref{thm:fromLimit}} shows how it is possible to use another problem to solve the problem of Principal with multiple Agents in cooperation. The other problem here is a McKean--Vlasov control problem. Notice that the solution $(\gamma,\aleph)$ of $\widehat{V}$ needs to satisfy more conditions than just the admissibility requirement of $\Ac$. We can avoid these conditions but the construction of the sequence of contracts would become less straightforward for the reader.

    \medskip
    $(ii)$ It is worth mentioning that, despite allowing a general form for the contracts i.e. $\xi^n$ and $\alephbb^n$ are just Borel maps of the $n$--production $(X^{1,n},\cdots,X^{n,n})$, this result proves that it is approximately optimal to only consider contracts that are maps of the empirical distribution of $(X^{1,n},\cdots,X^{n,n})$ as in {\rm \cite{MFD_PA_comp}}. However, in contrast to {\rm \cite{MFD_PA_comp}}, the form is even more specific.

    \medskip
    $(iii)$ Another interesting observation of this result is the absence of having to solve a stochastic control problem involving the control of the volatility in order to solve the problem of the principal. Indeed, as shown in {\rm \cite{cvitanic2015dynamic}}, we need sometimes to solve a stochastic control problem with control of volatility to obtain the optimal contract. Also, obtaining the approximately optimal contract does not involve the map $U$. As we can see the limit problem does not take into account the map $U$.
\end{remark}

\subsection{Connection to the Principal--Agent problem with McKean--Vlasov dynamics } 
\label{section:PA_McKean}
As stated in the introduction, the previous problem of the principal with $n$ agents is connected to another Principal--Agent problem. Namely, the Principal--Agent problem with McKean--Vlasov dynamics. In this section, we present the mathematical formulation of the Principal--Agent problem with McKean--Vlasov dynamics.

\paragraph*{Admissible contracts} 
The set of contracts are $\Cf=\left(\overline{\xi},\aleph \right) \in \Xi$ where $\aleph:[0,T] \x \R \to \Er$ and $\overline{\xi}:\Pc_p(C([0,T];\R)) \to \R$ are Borel maps and, there exist Borel maps $\xi:C([0,T];\R) \x \Pc_p(C([0,T];\R)) \to \R$ and $\psi:C([0,T];\R) \to \R_+$ s.t. $g \left(m, \overline{\xi}(m) \right)=\int_{C([0,T];\R)} \xi(z,m),m(\mathrm{d}z)$, for each $x$, $m \mapsto \xi(x,m)$ is continuous and
\begin{align*}
    \sup_{x,m}\frac{\left|\xi(x,m) \right|}{1 + \psi(x) + \|m\|_p^p} < \infty,
\end{align*}
where the supremum is taken over $x \in C([0,T];\R)$ and $m \in \Pc_p(C([0,T];\R))$ with $\int_{C([0,T];\R)} |\xi(z,m)|m(\mathrm{d}z)< \infty$.
Furthermore,  $\aleph$ and the map $\psi$ satisfy $\E\left[\int_0^T|\aleph(t,U_t)|^{2}\mathrm{d}t + \psi(U)^2 \right] < \infty$ where $U_0=\iota$, $\mathrm{d}U_t=\sigma(t,U_t)\mathrm{d}W_t$. We will usually refer to $\xi$ instead of $\overline{\xi}$.

\paragraph*{Admissible controls} Given a contract $\Cf$ associated to $\psi$, we say $\alpha \in \Ac(\Cf)$ is an admissible control associated to $\Cf$ if: $\alpha:[0,T] \x C([0,T];\R) \to A$ is progressively Borel measurable, the process $X^\alpha$ is well defined with $\mub^\alpha=\Lc(X^\alpha)$, $\mu=(\mu^\alpha_t=\Lc(X^\alpha_t))_{t \in [0,T]}$, $X^\alpha_0=\iota$,
\begin{align*}
    \mathrm{d}X^{\alpha}_t
    =
    b\left(t,X^\alpha_t,\mu^\alpha,\aleph(t,X^\alpha_t), \alpha(t,X^\alpha) \right) \mathrm{d}t + \sigma(t,X^{\alpha}_t) \mathrm{d}W_t.
\end{align*}
In addition, $\E \left[ \psi(X^\alpha) + \int_0^T \left| \alpha(t,X^\alpha) \right|^2 \mathrm{d}t \right] < \infty$ and   $\E\left[|L^{\alpha}_T|^2 \right]< \infty$ where $L^{\alpha}_0=1$, $U_0=\iota,$ 
\begin{align*}
    \mathrm{d}L^{\alpha}_t
    =
    L^{\alpha}_t \sigma(t,U_t)^{-1}b\left(t, U_t, \mu^\alpha,\aleph(t,U_t), \alpha(t,U) \right) \mathrm{d}W_t\;\mbox{and}\;\mathrm{d}U_t=\sigma(t,U_t)\mathrm{d}W_t.
\end{align*}

\medskip
We will say $\Cf=(\xi,\aleph) \in \widehat{\Xi} \subset \Xi$ if there exists a Borel map $\gamma:[0,T] \x \R \to \Er$ and $Y_0 \in \R$ s.t.
\begin{align*}
    \E \left[\exp \left(\int_0^T a |\aleph(t,U_t)|^2 + b |\gamma(t,U_t)|^2 \mathrm{d}t \right) \right]< \infty,\;\;\mbox{for each }a,b \ge 0,
\end{align*}
with in addition, for any $\alpha \in \Ac(\Cf)$,
\begin{align*}
    g \left(\mu^\alpha,\overline{\xi}(\mub^\alpha) \right)
    =
    Y_0
    -
    \E \left[\int_0^T  H (t,X^\alpha_t, \mu^\alpha,\aleph(t,X^\alpha_t), \gamma(t,X^\alpha_t)) \; \mathrm{d}t \right] + \E\left[ \int_0^T \gamma(t,X^\alpha_t) \sigma(t,X^\alpha_t)^{-1} \mathrm{d}X^\alpha_t \right].
\end{align*}

\begin{remark}
    Even though we consider the requirements for contracts and controls to be context--appropriate, due to the measure dependence they may not seem natural at first glance. In the absence of measure dependence i.e. classical Principal--Agent setting, these conditions are enough to establish most equivalence results of the literature under the assumptions we will consider.
\end{remark}


\medskip
For a contract $\Cf$, the agent uses $\alpha \in \Ac(\Cf)$ to control the production $X^\alpha$ in order to solve $V_A^{\Cf}:=\sup_{\alpha \in \Ac(\Cf)} J_A^{\Cf}(\alpha)$ where 
\begin{align*}
    J_A^{\Cf}(\alpha)
    :=
    \E \left[ 
    \int_0^T L\left(t,X^\alpha_t,\mu^\alpha,\aleph(t,X^\alpha_t),\alpha(t,X^\alpha) \right)\;\mathrm{d}t + g \left(\mu^\alpha,\overline{\xi}(\mub^\alpha) \right)
    \right].
\end{align*}
We denote by $\Ac^\star(\Cf)$ the set of admissible controls solving $V^{\Cf}_A$ and with the reservation utility $R$, we write ${\rm PE}[\Cf]$ the set of admissible controls solving $V^{\Cf}_A$ with $V^{\Cf}_A \ge R$.

\medskip
The problem of the Principal is 
\begin{align} \label{eq:value_PA_McKean}
    \overline{V}_{\rm P} := \sup_{\Cf \in \Xi} \sup_{\alpha \in {\rm PE}[\Cf] } J_{\Pr}^{\alpha}(\Cf),\;\; J_{\Pr}^{\alpha}(\Cf):=U \left(\E \left[ \Upsilon(X^\alpha_T) - g_{{\rm P}}\left(\mu^\alpha, \overline{\xi}(\mub^\alpha)  \right) - \int_0^T L_{\Pr} \left(t, \aleph(t,X^\alpha_t)\right)\;\mathrm{d}t  \right] \right).
\end{align}
By restricting the set of contracts on $\widehat{\Xi}$, we also introduce ${V}_{\Pr}:=\sup_{\Cf \in \widehat{\Xi}} \sup_{\alpha \in {\rm PE}[\Cf] } J_{\Pr}^{\alpha}(\Cf)$.


\begin{theorem} \label{thm:PA_McK_reduced}
    Under {\rm \Cref{assum:main1}},  {\rm \Cref{assump:growth_coercivity}} and {\rm \Cref{assump:growth_coercivity1}}, we have that 
    $$
        {V}_{\Pr}=U(\widehat{V} )=\lim_{n \to \infty} V_{n,\Pr}. 
    $$
\end{theorem}

\begin{remark}
    {\rm \Cref{thm:PA_McK_reduced}} gives another perspective on the limit of the problem of Principal with multiple Agents in cooperation when the number of agents goes to infinity. The limit is actually equal to a Principal--Agent problem of a new kind.
\end{remark}

Similarly to the problem of Principal with multiple Agents, we provide a constructive way to obtain an optimal contract for the problem $V_{\Pr}$ from the McKean--Vlasov control problem $\widehat{V}$. Let $(\gamma,\aleph) \in \Ac$ be an optimal control for $\widehat{V}$ satisfying the condition of \Cref{eq:opti_control_cond}. We define $\overline{\xi}(\mub):=g^{-1}\left(\mu, \int \xi(x,\mub) \mub(\mathrm{d}x) \right)$ where
\begin{align*}
     \xi(X,\mub)
    :=
    Y_0^{\gamma,\aleph} -
    \int_0^T  H (t,X_t, \mu,\aleph(t,X_t), \gamma(t,X_t)) \; \mathrm{d}t  + \int_0^T \gamma(t,X_t) \sigma(t, X_t)^{-1} \mathrm{d}X_t,
\end{align*}
for any $X$ s.t. $\Lc(X)=\mub$, $\mu=(\Lc(X_t))_{t \in [0,T]}$, and $X$ is a semi--martingale verifying $\int_0^T \gamma(t,X_t)^2 \sigma(t,X_t)^{-2} \mathrm{d}\langle X \rangle_t < \infty$ a.e.
\begin{proposition} \label{prop:PA_McKean_construction}
    Let {\rm \Cref{assum:main1}},  {\rm \Cref{assump:growth_coercivity}} and {\rm \Cref{assump:growth_coercivity1}} be true. The contract $\Cf=(\xi,\aleph)$ belongs to $\widehat{\Xi}$, it is an optimal contract for the problem of the Principal $V_{\Pr}$ i.e.
    \begin{align*}
        {V}_{\Pr}=U(\widehat{V} ) = \sup_{\alpha \in {\rm PE}[\Cf] } J_{\Pr}^{\alpha}(\Cf).
    \end{align*}
    The optimal control of the problem of the agent is given by $(t,x) \mapsto \widehat{\alpha}(t,x_t,\mu,\aleph(t,x_t),\gamma(t,x_t))$ where $\mu= (\mu_t=\Lc(X_t))_{t \in [0,T]}$ with $X_0=\iota$ and
    \begin{align*}
        \mathrm{d}X_t =  \widehat{b}\left(t,X_{t},\mu,\aleph(t,X_t),\gamma(t,X_t)  \right) \;\mathrm{d}t 
        +
        \sigma(t,X_t) \mathrm{d}W_t.
    \end{align*}
\end{proposition}

\medskip
So far, we have considered the Principal--Agent problem with McKean--Vlasov dynamics when the set of contracts is in $\widehat{\Xi}$. This restriction is actually not necessary as we will now see.
Let us denote by $\Mc$ the set of progressively Borel measurable maps $\Gamma:[0,T] \x C([0,T];\R) \to \Pc(\R)$ and by $\Mc_0$  the set of progressively Borel measurable maps $\gamma:[0,T] \x C([0,T];\R) \to \R$.

\begin{proposition} \label{prop:PA_Mckean_charac}
    Let {\rm \Cref{assum:main1}},  {\rm \Cref{assump:growth_coercivity}} and {\rm \Cref{assump:growth_coercivity1}} be true, $b$ be bounded and $\Cf=(\xi,\aleph) \in \Xi$ be  a contract.

\medskip
    $(i)$ For any $\alpha \in \Ac(\Cf)$, there exists $\Gamma \in \Mc$ s.t.
    \begin{align*}
        J^{\Cf}_A(\alpha)
        \le 
        &\sup_{\alpha' \in \Ac(\Cf)} J^{\Cf}_A(\alpha')
        \\
        &~~+ \E \left[ \int_0^T \int_{\R} h\left(t,X^{\alpha}_t,\mu^\alpha,\aleph(t,X^\alpha_t),z,\alpha(t,X^\alpha) \right)-  H\left(t,X^{\alpha}_t,\mu^\alpha,\aleph(t,X^\alpha_t),z \right) \Gamma(t,X^\alpha)(\mathrm{d}z)\mathrm{d}t \right].
    \end{align*}

\medskip
$(ii)$ If $\alpha \in \Ac(\Cf)$ is optimal i.e. $\alpha \in \Ac^\star(\Cf)$, there exists $\gamma \in \Mc_0$ s.t. $\mathrm{d}t \otimes \mathrm{d}\P$--a.e. $$
    \alpha(t,X^\alpha)=\widehat{\alpha}(t,X^\alpha_t,\mu^\alpha, \aleph(t,X^\alpha_t), \sigma(t,X^\alpha_t)^{-1} \gamma(t,X^\alpha) )
$$
and 
\begin{align*}
    g \left(\mu^\alpha,\overline{\xi}(\mub^\alpha) \right)
    =
    \sup_{\alpha' \in \Ac(\Cf)} J^{\Cf}_A(\alpha')
    -
    \E \left[\int_0^T  \widehat{L} (t,X^\alpha_t, \mu^\alpha,\aleph(t,X^\alpha_t), \gamma(t,X^\alpha)) \; \mathrm{d}t \right]. 
\end{align*}

\end{proposition}

\begin{remark}
    $(i)$ {\rm \Cref{prop:PA_Mckean_charac}} gives us a characterization of the optimal control of the problem faced by the agent. Putting aside the Principal--Agent problem aspect of the formulation, this proposition gives results regarding the McKean--Vlasov control problem that are never been established before under the assumptions we assumed. Namely, the map $(t,x,m) \mapsto \left(\aleph(t,x), \xi(x,m)\right)$ is just Borel measurable in $x$ and continuous in $m$. 

\medskip
    $(ii)$ The condition $``$$b$ is bounded $"$ can be weaken and replaced by some integrability condition. We chose it for simplicity and to avoid to lose the reader through unnecessary technicalities. 
\end{remark}

\begin{theorem}\label{thm:PA_McK_general}
    Under {\rm \Cref{assum:main1}},  {\rm \Cref{assump:growth_coercivity}} and {\rm \Cref{assump:growth_coercivity1}}, with $b$ bounded, we have 
$$  
    \overline{V}_{\Pr}=V_{\Pr}=U(\widehat{V} ).
$$    
\end{theorem}

\begin{remark}
    {\rm \Cref{thm:PA_McK_general}} actually shows that, even if we allow a more general set of contracts, it is actually optimal to only consider contracts over the set $\widehat{\Xi}$. Combining this result and the convergence of the problem of Principal with multiple Agents, by classical techniques $($see for instance  {\rm \cite{MFD_PA_comp}}$)$, we can show that the limit of any sequence of optimal contracts for the problem of Principal with $n$ Agents is related to a contract or sequence of contracts belonging to $\widehat{\Xi}$ when $n$ goes to infinity. Also, by using similar techniques to {\rm \Cref{thm:fromLimit}}, any contract of $\widehat{\Xi}$ can be approximated by a sequence of contracts admissible for the problem of Principal with $n$ Agents. 
\end{remark}

\subsection{Connection to the problem of multitask Principal--Agent} \label{sec:multitask}

The multitask Principal--Agent problem has been presented and explored by \citeauthor*{Holmstrom_Milgrom_multitask} in \cite{Holmstrom_Milgrom_multitask}. As explained in the introduction, the Multitask Principal--Agent problem occurs when a principal assigns multiple tasks that can be correlated to an agent. The challenge lies in designing incentives that motivate the agent to allocate effort across all tasks efficiently. We revisit here a simple example and see his connection with the previous problems examined. We will focus on the impact of the correlation/interdependence of the tasks. The model is as follow: with a contract $\Cf^n=\xi^n$, the agent is facing the problem
\begin{align*}
    \sup_{\alphab \in \Ac_n^n}\E \left[ \xi^n(\Xbb^{\alphab}) - \frac{1}{2n} \sum_{i=1}^n \int_0^T |\alpha^{i,n}(t,\Xbb^{\alphab})|^2 \mathrm{d}t \right]\;\mbox{with}\;\mathrm{d}X^{i,\alphab}_t=\Big( \alpha^{i,n}(t,\Xbb^{\alphab}) + \frac{\overline{\kappa}}{n}\sum_{j=1}^n b(X^{j,\alphab}_t ) \Big)\;\mathrm{d}t + \mathrm{d}W^i_t
\end{align*}
where $\Xbb^{\alphab}=(X^{1,\alphab},\dots,X^{n,\alphab})$ is the production. The map $b$ is given by $b(x)= -\overline{b} \vee (\overline{b} \wedge x)$ with $\overline{b}>0$. The number of tasks $n$ and the bound $\overline{b}$ are very large. The dynamics of the production of one task $X^i$ is impacted by the production of all the tasks through the component $\frac{\overline{\kappa}}{n}\sum_{j=1}^n b(X^j_t )$. The Principal is then trying to solve
\begin{align*}
   V_{n,\Pr}=\sup_{\xi^n} \sup_{\alphab \in {\rm PE}[\Cf^n]} \E \left[ U \left( \frac{1}{n} \sum_{i=1}^n X^{i,\alphab}_T - \xi^n(\Xbb^{\alphab}) \right) \right].
\end{align*}
Although the interpretation and the meaning of the equations are different, from a mathematical point of view, the previous formulation falls into the framework of the problem of principal with multiple agents in cooperation of \Cref{ssec:equidyn}. Let us then use the approach developed in \Cref{ssec:equidyn}.  The map $\widehat{\alpha}$ is given by $\widehat{\alpha}(t,x,m,e,z)=\alpha^\star(t,z)$ where
\begin{align*}
    \alpha^\star(t,z)
    :=
    z.
\end{align*}

According to what we proved in \Cref{thm:cong_PA_comp}, the corresponding McKean--Vlasov problem is:
\begin{align*}
    \widehat{V}=\sup_{\gamma} \widehat{J}(\gamma)\;\mbox{with}\;\widehat{J}(\gamma):=\E \left[ X_T - R - \int_0^T \frac{1}{2} \gamma(t,X_t)^2 \mathrm{d}t \right]\mbox{ where }\mathrm{d}X_t=\Big( \gamma(t,X_t) + \overline{k} \E[b(X_t)] \Big)\mathrm{d}t + \sigma \mathrm{d}W_t.
\end{align*}
We define the controls $\widehat{\gamma}(t,x):=e^{\overline{k}(T-t)}$, $\widehat{\alphab}^n:=(\widehat{\alpha}^{1,n}, \cdots,\widehat{\alpha}^{n,n})$ with $\widehat{\alpha}^{i,n}(t,x^1,\cdots,x^n)=e^{\overline{k}(T-t)}$, and the contract $\Cf^n:=\xi^n$, for any $\alphab$
\begin{align*}
    \xi^n(\Xbb^{\alphab})
    :=
    R -  \int_0^T \frac{1}{2}|\widehat{\gamma}_t|^2 + \frac{\overline{\kappa}}{n} \sum_{i=1}^n b(X^{i,\alphab}_t) \mathrm{d}t + \frac{1}{n} \sum_{i=1}^n \int_0^T \widehat{\gamma}_t \mathrm{d}X^{i,\alphab}_t. 
\end{align*}

\begin{proposition} \label{prop:multitask}
    We have that $\lim_{n \to \infty} V_{n,p}=U(\widehat{V})$. If we assume that $U$ is Lipschitz, there exists a positive constant $C$ independent of $n$ and $\overline{b}$ s.t.
    \begin{align*}
        V_{n,\Pr}- J^{\hat{\alphab}^n}_{n,\Pr}(\Cf^n) \le C ( n^{-1/2} + \overline{b}^{-1})\;\mbox{and}\;\widehat{V}-\widehat{J}(\widehat{\gamma}) \le C \overline{b}^{-1}.
    \end{align*}
    The contract $\xi^n$ satisfies $\xi^n ( \Xbb^{\hat \alphab^n} )= R + \frac{1}{2} \int_0^T |\widehat{\gamma}_t|^2 \mathrm{d}t + \frac{1}{n} \sum_{i=1}^n \int_0^T \widehat{\gamma}_t \mathrm{d}W^i_t$ and
    \begin{align*}
        J^{\hat \alphab^n}_{n,\Pr}(\Cf^n)=\E \left[ U \left( - R + e^{\overline{k}T}\frac{1}{n} \sum_{i=1}^n \iota^i + \frac{1}{2} \int_0^T |\widehat{\gamma}_t|^2 \mathrm{d}t + \frac{2}{n} \sum_{i=1}^n \int_0^T \widehat{\gamma}_t \mathrm{d}W^i_t  \right) \right].
    \end{align*}
\end{proposition}

\begin{remark}
    $(i)$ Since the number of tasks $n$ and the bound $\overline{b}$ are taken very large, the contract $\xi^n$ given in the proposition provide a good approximation for the problem of the Principal.

\medskip
    $(ii)$ We chose to stay in this simple situation to highlight the important of the correlation/interdependence parameter $\overline{\kappa}$. The case $\overline{\kappa}=0$ means no correlation between the tasks. This result shows that the more positive and larger $\overline{\kappa}$ is, the greater the utility of the principal. Conversely, the more negative and smaller $\overline{\kappa}$ is, the smaller the utility of the principal. The variance of the contract is small since $n$ is considered very large.
    
\medskip
    $(iii)$ The case $\overline{\kappa}=0$ can be seen as a situation where the principal and the agent are unaware of the correlation/interdependence of the tasks due to lack of technological knowledge for instance. By offering a contract under the assumption that $\overline{\kappa}=0$, the principal is actually reducing his utility if the true $\overline{\kappa}>0$ is positive. In contrast, the principal is increasing his utility if the true $\overline{\kappa} < 0$ is negative.
\end{remark}

\section{Proofs of the main results} \label{sec:proofs}

\subsection{Convergence in the problem of Principal with multiple agents}

This section is devoted to the proof of \Cref{thm:cong_PA_comp} and \Cref{thm:fromLimit}. We will start in the next section by providing a characterization through BSDEs of the $n$--agent problem.
\paragraph*{BSDE characterization of the $n$--agent problem} Recall that the map $\widehat{\alpha}(t,x,\pi,e,z)$ was introduced in \eqref{eq:maximizee} for any $(t,x,\pi,e,z)$.
We remind the maps $h(t,x,\pi,e,z,u):=b \left(t,x,\pi,e,u \right) \sigma^{-1}(t,x)z +  L \left(t,x,\pi,e,u \right)$ and $H(t,x,\pi,e,z)=\widehat{L}(t,x,\pi,e,z) + \sigma^{-1}(t,x) z \;\widehat{b}(t,x,\pi,e,z).$
We denote by $\S$ the set of $\R$--valued continuous $\F$--adapted process $S$ verifying: $\E\left[ \exp\left( a \sup_{0 \le t \le T} |S_t| \right) \right] < \infty$ for any $a \ge 0$. We write $\H$ the set of $\R$--valued $\F$--progressively measurable $K$ s.t. $\E\left[ \left(\int_0^T |K_t|^2 \mathrm{d}t \right)^{q/2} \right] < \infty$ for any $q \ge 1$. We also consider $\H_{\exp}$ the set of $\R$--valued $\F$--progressively measurable $K$ s.t. $\E \left[ \exp \left( b\;\int_0^T|K_t| \mathrm{d}t \right) \right] < \infty$ for any $b \ge 0$.

\medskip
Let $\Cf^n=(\xi^n,\alephbb^n)$ be a contract. We consider the processes $(Y^n,\Zbb^n:=(Z^{1,n},\cdots,Z^{n,n}),\Xbb^n)$ satisfying: $(Y^n,\Zbb^n) \in \S \x \H^n$, $(Y^n,\Zbb^n,\Wbb^n)$ is $\left(\sigma \left( \Xbb^n_s,\;s \le t \right) \right)_{t \le T}$--adapted, for all $t \in [0,T],$ a.e. 
\begin{align} \label{eq:BSDE1_agent}
    Y^n_t
    =
    g\left( \mu^n,\xi^n(\Xbb) \right)
    +
    &\frac{1}{n} \sum_{i=1}^n \int_t^T  \widehat{L}\left(t,X^{i}_s,,\mu^n,\aleph^{i,n}(s,\Xbb^n),nZ^{i,n}_s \right)\;\mathrm{d}s 
    - \sum_{i=1}^n\int_t^T Z^{i,n}_s \mathrm{d}W^i_s
\end{align}
with ${\mu}^n_t:=\frac{1}{n} \sum_{i=1}^n \delta_{X^{i,n}_t}$
and for each $1 \le i \le n$,
\begin{align} \label{eq:BSDE2_agent}
    X^{i,n}_t = \iota^i + \int_0^t  \widehat{b}\left(s,X^i_{s}, \mu^n ,\aleph^{i,n}(s,\Xbb^n),n Z^{i,n}_s \right) \;\mathrm{d}s 
        +
        \int_0^t \sigma(s,X^i_s) \mathrm{d}W^i_s.
\end{align}
We set $\alpha^{\star,n}:=(\alpha^{\star,1,n},\cdots,\alpha^{\star,n,n})$ by
\begin{align*}
    \alphab^{\star,i,n}(t,\Xbb^n):=\widehat{\alpha}\left(t,X^i_{t}  , \mu^n,\aleph^{i,n}(t,\Xbb^n),,\sigma(t,X^i_t)^{-1}\;n Z^{i,n}_t\right).
\end{align*}

\begin{proposition} \label{prop:characterization_agent}
    Under {\rm\Cref{assum:main1}}, {\rm \Cref{assump:growth_coercivity}} and {\rm \Cref{assump:growth_coercivity1}}, given a contract $\Cf^n=(\xi^n, \aleph^n)$ and an equilibrium  $\alphab:=(\alpha^{1,n},\cdots,\alpha^{n,n})$ i.e. belonging to $\Ac^{\star,n}_n(\Cf^n)$, then $(Y^n,\Zbb^n,\Xbb^n)$ is well--defined, 
    \begin{align*}
        \alpha^{i,n}\left( t, \Xbb^n\right)
        =
        \widehat{\alpha}\left(t,X^i_{t}  , \mu^n,\aleph^{i,n}(t,\Xbb^n),\sigma(t,X^i_t)^{-1}\;n Z^{i,n}_t\right)\;\mathrm{d}t \otimes\mathrm{d}\P\mbox{--a.e. and }\E\left[\Rr^{\Cf^n}_n(\alphab) \big| \Xbb^n_0 \right]=Y^n_0.
    \end{align*}
    Also, if $(Y^n,\Zbb^n,\Xbb^n)$ is well--defined and $\alphab^{\star,n}$ is admissible i.e. belongs to $\Ac^n_n$, then $\alphab^{\star,n} \in \Ac^{\star,n}_n(\Cf^n)$.
\end{proposition}
\begin{proof}

\medskip
Let us define the processes $(U^i)_{1 \le i \le n}$ as the solution of: $U^i_0=\iota^i$ and $\mathrm{d}U^i_t=\sigma(t,U^i_t) \mathrm{d}W^i_t$. Let $\alphab \in \Ac^{\star,n}_n(\Cf^n)$,  under \Cref{assum:main1}, by using \cite[Corollary 6]{briand2008quadratic}, there exists a unique solution $(\widehat{Y}^n, \widehat{\Zbb}^n)$ satisfying: $(\widehat{Y}^n, \widehat{\Zbb}^n) \in \S \x \H^n$, $\widehat{Y}^n_T=g (\nu^n, \xi^n(\Ubb^n))$
\begin{align} \label{eq:bsde_optimal}
    \mathrm{d} \widehat{Y}^n_t=
    - \frac{1}{n} \sum_{i=1}^n h (t, U^i_t, \nu^n, \aleph^{i,n}(t, \Ubb^n ), n \widehat{Z}^i_t, \alpha^i(t,\Ubb^n)) \mathrm{d}t + \sum_{i=1}^n \widehat{Z}^i_t \mathrm{d}W^i_t\;\mbox{with}\;\nu^n_t=\frac{1}{n} \sum_{i=1}^n \delta_{U^i_t}.
\end{align}
\medskip    
    Let $\boldsymbol{\beta} \in \Ac_n^n$ be an admissible controls. We introduce the probability $\P^{\boldsymbol{\beta}}$ through density process $L^{n,\betab}$ defined by
    \begin{align*}
        \frac{\mathrm{d \P^{\boldsymbol{\beta}}}}{\mathrm{d}\P}:=L^{n,\betab}_T\;\;\mbox{with}\;\;\frac{\mathrm{d}L^{n,\betab}_t}{L^{n,\betab}_t}= \sum_{i=1}^n \phi^i(t,\Ubb^n)\;\mathrm{d}W^i_t
    \end{align*}
    and 
    \begin{align*}
        \phi^i(t,\Ubb^n)
        :=
        \sigma^{-1}(t,U^i_t) b\left(t,U^i_{t}, \nu^n ,\aleph^{i,n}(t,\Ubb^n),\beta^{i,n}(t,\Ubb^n) \right).
    \end{align*}
    Notice that, by using the assumption over $b$ and the fact that $\betab$ is admissible,  by Girsanov Theorem, we have that $W^{\beta,i}_\cdot=W^i_\cdot-\int_0^\cdot \phi^i(s,\Xbb_s) \mathrm{d}s$ is a $\P^{\boldsymbol{\beta}}$--Brownian motion, 
    and $\Lc \left( \Xbb^{\boldsymbol{\beta},n} \right)=\Lc^{\P^{\boldsymbol{\beta}}}( \Ubb^n)$.
We introduce the variables $(V_t)_{t \in [0,T]}$ by
\begin{align*}
    V_{t}
    :=
    \esup_{\betab \in \Ac_n^n}\E^{\P^{\betab}} \left[ \frac{1}{n} \sum_{i=1}^n
        \int_t^T 
        L\left(s,U^{i}_s,\nu^n,\aleph^{i,n}(s,\Ubb^n),\beta^{i}(s,\Ubb^n) \right)\;\mathrm{d}s 
        +
        g\left( \nu^n,\xi^n(\Ubb^n) \right) \bigg| \Fc_t
        \right].
\end{align*}
We can check that $\E[V_0]=\sup_{\betab \in \Ac_n^n} J^{\Cf^n}_{n}(\betab^n)$. By some classical dynamic programming principle (see for instance \cite[Theorem 2.1]{karoui2013capacities2}), we can verify that: for any $(\sigma(\Ubb^n_s:\;s\le t))_{t \in [0,T]}$--stopping time $\tau$ verifying $t \le \tau \le T$,
\begin{align*}
    V_t 
    =
    \esup_{\betab \in \Ac_n^n}\E^{\P^{\betab}} \left[ \frac{1}{n} \sum_{i=1}^n
        \int_t^\tau 
        L\left(s,U^{i}_s,\nu^n,\aleph^{i,n}(s,\Ubb^n),\beta^{i}(s,\Ubb^n) \right)\;\mathrm{d}s 
        +
        V_\tau \bigg| \Fc_t
        \right].
\end{align*}
For any admissible $\betab \in \Ac_n^n$, we define the process $(M^{\betab}_t)_{t \in [0,T]}$ by
\begin{align*}
    M^{\betab}_t
    :=
    V_t + \frac{1}{n} \sum_{i=1}^n \int_0^t L\left(s,U^{i}_s,\nu^n,\aleph^{i,n}(s,\Ubb^n),\beta^{i}(s,\Ubb^n) \right)\;\mathrm{d}s.
\end{align*}
We check that $(M^{\betab}_t)_{t \in [0,T]}$ is a $\P^{\betab}$--super martingale for any $\betab$, and by using the fact that $\alphab$ is an equilibrium and the martingale optimality principle, $(M^{\alphab}_t)_{t \in [0,T]}$ is a $\P^{\alphab}$--martingale. We can verify that: for any admissible $\betab \in \Ac_n^n$, $\E^{\P^{\betab}}[\int_0^T |\widehat{Z}^i_t|^2\mathrm{d}t]< \infty$, therefore $\E^{\P^{\betab}}[\int_t^T \widehat{Z}^i_r \mathrm{d}W^{\beta,i}_r|\Fc_t]=0$ and by classical verification arguments  
$$
    \widehat{Y}^n_t=V_t=\E^{\P^{\alphab}} \left[ \frac{1}{n} \sum_{i=1}^n
        \int_t^T 
        L\left(s,U^{i}_s,\nu^n,\aleph^{i,n}(s,\Ubb^n),\alpha^{i}(s,\Ubb^n) \right)\;\mathrm{d}s 
        +
        g\left( \nu^n,\xi^n(\Ubb^n) \right) \bigg| \Fc_t
        \right].
$$
Consequently, for any admissible $\betab \in \Ac_n^n$,
\begin{align*}
    M^{\betab}_t 
    &=
    V_0 - \frac{1}{n} \sum_{i=1}^n \int_0^t h(s,U^i_s,\nu^n,\aleph^{i,n}(s,\Ubb^n), n \widehat{Z}^i_s, \alpha^i(s,\Ubb^n)) \mathrm{d}s 
    \\
    &~~~~~~~~~~+ \frac{1}{n} \sum_{i=1}^n\int_0^t L(s,U^i_s, \nu^n, \aleph^{i,n}(s,\Ubb^n), \beta^i(s,\Ubb^n)) \mathrm{d}s + \sum_{i=1}^n \widehat{Z}^i_s \mathrm{d}W^i_s
    \\
    &=
    V_0 - \frac{1}{n} \sum_{i=1}^n \int_0^t h(s,U^i_s,\nu^n,\aleph^{i,n}(s,\Ubb^n), n \widehat{Z}^i_s, \alpha^i(s,\Ubb^n)) \mathrm{d}s 
    \\
    &~~~~~~~~~~+ \frac{1}{n} \sum_{i=1}^n\int_0^t h(s,U^i_s, \nu^n, \aleph^{i,n}(s,\Ubb^n), n \widehat{Z}^i_s, \beta^i(s,\Ubb^n)) \mathrm{d}s + \sum_{i=1}^n \widehat{Z}^i_s \mathrm{d}W^{\betab,i}_s.
\end{align*}
Since $(M^{\betab}_t)_{t \in [0,T]}$ is a $\P^{\betab}$--super martingale for any $\betab$, we have necessarily that $\mathrm{d}s \otimes \mathrm{d}\P^{\betab}$--a.e.
\begin{align*}
     \frac{1}{n} \sum_{i=1}^n h(s,U^i_s,\nu^n,\aleph^{i,n}(s,\Ubb^n), n \widehat{Z}^i_s, \alpha^i(s,\Ubb^n)) \ge \frac{1}{n} \sum_{i=1}^n h(s,U^i_s, \nu^n, \aleph^{i,n}(s,\Ubb^n), n \widehat{Z}^i_s, \beta^i(s,\Ubb^n)).
\end{align*}
The probabilities $(\mathrm{d}s \otimes \mathrm{d}\P^{\betab})_{\betab \in \Ac_n^n}$ are equivalent, then the previous result is true $\mathrm{d}s \otimes \mathrm{d}\P$--a.e. for any $\betab$. We can conclue that for any $i$, a.e. $s \in [0,T]$, any $\betab \in \Ac_n^n$,
\begin{align*}
     h(s,U^i_s,\nu^n,\aleph^{i,n}(s,\Ubb^n), n \widehat{Z}^i_s, \alpha^i(s,\Ubb^n)) \ge h(s,U^i_s, \nu^n, \aleph^{i,n}(s,\Ubb^n), n \widehat{Z}^i_s, \beta^i(s,\Ubb^n)).
\end{align*}
By uniqueness of the maximizer of the Hamiltonian $h$, we deduce that 
$$
    \alpha^i(t,\Ubb^n)=\widehat{\alpha}(t,U^i_t,\nu^n,\aleph^{i,n}(t,\Ubb^n), \sigma(t,U^i_t)^{-1} n \widehat{Z}^i_t).
$$
We easily check that $V_0=\widehat{Y}^n_0=\E\left[\Rr^{\Cf^n}_n(\alphab^n) \big| \Xbb^n_0 \right]$. To conclude, it is enough to notice that the well defined distribution $\Lc^{\P^{\alphab}}(\widehat{Y}^n, \widehat{\Zbb}^n, \Ubb^n)$ is equal to the distribution of $\Lc \left( Y^n, \Zbb^n, \Xbb^n \right)$. The second part of the Proposition follows by similar arguments as before. 
    
\end{proof}

\paragraph*{Characterization of the contracts through SDEs} Now, we give some characterization of the contracts. We recall that $\widehat{g}_{\Pr}(\pi,y)=g_{\Pr}(\pi, g^{-1}(\pi,y)),\;\mbox{with}\;g^{-1}(\pi,\cdot)$ denotes the inverse of $g(\pi,\cdot)$ for any $\pi \in C([0,T];\Pc_p(\R)).$ 

\medskip
Given the reservation utility $R$, we will say the process $\If^n:=(\alephbb^n,\Zbb^n) \in \H^n_{\exp} \x \H^n$ belongs to $\Sc_n$ if there exits $(Y^{n,\If^n},\Xbb^{n,\If^n}):=(Y^n,\Xbb^n)$ s.t. $\Zbb^n$ is an $(\sigma(\Xbb^n_{t \wedge \cdot}))_{t \le T}$--progressively measurable, $Y^n\in \S$,  and $\alphab^n \in \Ac_n^n$ is admissible where 
$$
    \alpha^{i,n}(t,\Xbb^n)=\widehat{\alpha}(t,X^i_t,\mu^n,\aleph^{i,n}(t,\Xbb^n), \sigma(t,X^i_t)^{-1}n Z^{i,n}_t)
$$
and $\xi^n(\Xbb^n)=g^{-1}(\mu^n,Y^n_T)$ with for all $t \in [0,T],$ $Y^n_0 \ge  R$ a.e.,
\begin{align*}
    \mathrm{d}Y^n_t
    =
    -
    \frac{1}{n} \sum_{i=1}^n \widehat{L}\left(t,X^{i,n}_t,\mu^n,\aleph^{i,n}(t,\Xbb^n),n Z^{i,n}_t \right)\;\mathrm{d}t 
    + \sum_{i=1}^n Z^{i,n}_t \mathrm{d}W^i_t
\end{align*}
and for each $1 \le i \le n$, $\Xbb^n:=\Xbb^{\alphab^n}$ i.e. $X^{i,n}_0 = \iota^i$,
\begin{align*}
   \mathrm{d}X^{i,n}_t =  \widehat{b}\left(t,X^{i,n}_{t}, \mu^n ,\aleph^{i,n}(t,\Xbb^n), n Z^i_t \right) \;\mathrm{d}t 
        +
        \sigma(t,X^{i,n}_t) \mathrm{d}W^i_t.
\end{align*}
Notice that, for any $\If^n \in \Sc_n$, $Y^{n,\If^n}$ is $(\sigma(\Xbb^{n,\If^n}_s:s\le t))_{t \in [0,T]}$--adapted. Consequently, the couple $(\xi^n, \alephbb^n)$ can be seen as a contract where $\xi^n(\Xbb^{n,\If^n})=g^{-1}\left(\mu^n,Y^{n,\If^n}_T \right)$. We introduce 
\begin{align*}
    \widehat{J}_{n,\Pr} (\If^n)
    := 
        \E \left[ U \left( \frac{1}{n} \sum_{i=1}^n \;\;\Upsilon \left(X^{i,\If^n}_T\right)
    -\widehat{g}_{\Pr} \left(\mu^n,Y^{n,\If^n}_T \right)
    -
    \int_0^T L_{\Pr} \left(t,\aleph^{i,n}(t,\Xbb^{n,\If^n}) \right)\;\mathrm{d}t\right) \right]
\end{align*}
and
\begin{align*}
    \widehat{V}_{n,\Pr} := \sup_{\If^n=(\alephbb^n,\Zbb^n) \in \Sc_n}  \widehat{J}_{n,\Pr} (\If^n).
\end{align*}

\begin{proposition} \label{prop:reformulation_principal}
    Under {\rm\Cref{assum:main1}}, {\rm \Cref{assump:growth_coercivity}} and {\rm \Cref{assump:growth_coercivity1}}, $V_{n,\Pr}
        =
        \widehat{V}_{n,\Pr}.$
\end{proposition}
\begin{proof}
    The proof of this result is inspired by \cite[Theorem 4.2.]{cvitanic2015dynamic}. For any $\If^n \in \Sc_n$, we can check the required measurability and integrability to conclude that  $(\xi^n, \alephbb^n)$ is a contract where $\xi^n(\Xbb^{n,\If^n})=g^{-1}(\mu^n,Y^{n,\If^n}_T)$. Since the admissibility of  $\alphab^n$ is true by construction, by \Cref{prop:characterization_agent}, $\alphab^n$ belongs to ${\rm PE}[(\xi^n, \alephbb^n)]$. Consequently, we have $V_{n,\Pr}
        \ge 
        \widehat{J}_{n,\Pr} (\If^n).$ Conversely, let $\Cf^n=(\xi^n,\alephbb^n)$ be a contract associated to an equilibrium $\alphab^n \in {\rm PE}[\Cf^n]$, by using \Cref{prop:characterization_agent}, there exists $\If^n \in \Sc_n$ s.t. $\Xbb^{\alphab^n}=\Xbb^{n,\If^n}$, $\xi^n(\Xbb^{\alphab^n})=g^{-1}(\mu^n,Y^{n,\If^n}_T)$ and the optimal control is given by $\alpha^{i,n}(t,\Xbb^{\alphab^n})=\widehat{\alpha}(t,X^{i,\If^n}_t,\mu^n,\aleph^{i,n}(t,X^{i,\If^n}), \sigma(t,X^{i,\If^n}_t)^{-1} n Z^{i,n}_t)$. As a result, we have that $J_{n,\Pr}^{{\color{black}\alphab^n}}(\Cf^n) \le \widehat{V}_{n,\Pr}$. We can then conclude the proof of the proposition.

\end{proof}

\paragraph*{Convergence of the problem of the principal when $n$ goes to infinity} 
As we have shown in \Cref{prop:reformulation_principal}, with the reservation utility $R$ fixed, the problem of the principal is
\begin{align*}
    \widehat{V}_{n,\Pr} = \sup_{\If^n=(\alephbb^n,\Zbb^n) \in \Sc_n} \widehat{J}_{n,\Pr} (\If^n).
\end{align*}

In order to be able to take $n$ to infinity, we need to use the coercivity assumptions over our maps mentioned in \Cref{assump:growth_coercivity1}.
Under these assumptions, we have the following convergence result of the problem of the Principal to a McKean--Vlasov control problem. 
\begin{proposition} \label{prop:conv-principal}
    Under {\rm\Cref{assum:main1}}, {\rm \Cref{assump:growth_coercivity}} and {\rm \Cref{assump:growth_coercivity1}}, we have 
    \begin{align*}
        \lim_{n \to \infty} \sup_{\If^n=(\alephbb^n,\Zbb^n) \in \Sc_n} \widehat{J}_{n,\Pr} (\If^n) = U (\widehat{V})\;\;\mbox{where}\;\;\widehat{V}\;\mbox{ \rm is defined in \eqref{eq:limit_value}}.
    \end{align*}
\end{proposition}

\begin{proof}
    This result is essentially an application of \Cref{prop:relaxed_appr}, \Cref{prop:weak_appr} and \Cref{prop:markovian_appr} which are extensions of results in \cite{MFD-2020,MFD-2020-closed} in the case of unbounded coefficients and non--compact set valued of controls. Let $(\delta_n)_{n \ge 1}$ be a sequence of non--negative numbers satisfying $\lim_{n \to \infty} \delta_n=0$. We choose $\If^n \in \Sc_n$ s.t.
    \begin{align*}
         \widehat{J}_{n,\Pr} (\If^n) \ge \sup_{\If'^n=(\alephbb'^n,\Zbb'^n) \in \Sc_n} \widehat{J}_{n,\Pr} (\If'^n) - \delta_n.
    \end{align*}
    By \Cref{prop:coercivity}, we have that $ \sup_{n \ge 1} \E \left[ |Y^n_0| + \frac{1}{n} \sum_{i=1}^n \int_0^T |n\;Z^{i,n}_t|^{2} + |\aleph^{i,n}(t,\Xbb^{n,\If^n})|^{2} \mathrm{d}t \right] < \infty.$ Therefore
    \begin{align*}
        \lim_{n \to \infty} \sup_{n \ge 1} \E \left[ \sum_{i=1}^n \int_0^T|Z^{i,n}_t|^2 \mathrm{d}t \right]=0\;\mbox{leading to}\;\lim_{n \to \infty} \sup_{n \ge 1} \E \left[\left| \sum_{i=1}^n \int_0^TZ^{i,n}_t \mathrm{d}W^i_t \right|^2\right]=0.
    \end{align*}
    By combining \Cref{prop:relaxed_appr}, \Cref{prop:weak_appr} and \Cref{prop:markovian_appr}, the sequence $(\widehat{\Qr}^n)_{n \ge 1}$ is relatively compact for the weak convergence topology and bounded in $\Wc_1$ and, the sequence $(\Qr^n)_{n \ge 1}$ is relatively compact in $\Wc_r$ with $1 \le r < 2$ and bounded in $\Wc_2$ where $\widehat{\Qr}^n
    :=
    \frac{1}{n} \sum_{i=1}^n \P \circ \left( X^{i,\If^n}, W^i,\Lambda^i, \muh^n, Y^n_0 \right)^{-1}$ and $\Qr^n
    :=
    \frac{1}{n} \sum_{i=1}^n \P \circ \left( X^{i,\If^n}, W^i,\Lambda^i, \muh^n \right)^{-1},\;\Lambda^i:=\delta_{ (n Z^{i,n}_t, \aleph^{i,n}(t,\Xbb^{n,\If^n}))}(\mathrm{d}z,\mathrm{d}e)\mathrm{d}t,\;\muh^n:=\frac{1}{n}\sum_{i=1}^n \delta_{(X^{i,\If^n},W^i,\Lambda^i)}$. In addition, for any limit point $\widehat{\Qr}=\P \circ (X,W,\Lambda,\muh,Y_0)^{-1}$, there exists a sequence of bounded Lipschitz maps $(\gamma^{u,j},\aleph^{u,j})_{j \ge 1,u \in [0,1]} \subset \Ac$ s.t. $u \mapsto \widehat{J}(\gamma^{u,j},\aleph^{u,j})$ is Borel measurable and 
    \begin{align*}
        \lim_{j \to \infty} \int_0^1 U \left( \widehat{J}(\gamma^{u,j},\aleph^{u,j}) \right) \mathrm{d}u = \E \left[U \left( \E \left[  \Upsilon \left(X_T\right)
    -\widehat{g}_{\Pr} \left(\mu,Y_T \right)
    -
    \int_0^T \int_{\R \x \Er} L_{\Pr} \left(t,e \right)\;\Lambda_t(\mathrm{d}z,\mathrm{d}e)\mathrm{d}t \Big|\muh, Y_0
    \right] \right) \right]
    \end{align*}  
    where $Y_\cdot=Y_0 - \E \left[\int_0^\cdot \widehat{L} \left(t, X_t, \mu, e,z  \right) \Lambda_t(\mathrm{d}z,\mathrm{d}e) \mathrm{d}t \Big| \muh, Y_0\right]$ with $Y_0 \ge R$ a.e. Consequently, we have
    \begin{align*}
        &\limsup_{n \to \infty} \sup_{\If'^n=(\alephbb'^n,\Zbb'^n) \in \Sc_n} \widehat{J}_{n,\Pr} (\If'^n)
        \\
        &\le \limsup_{n \to \infty} \widehat{J}_{n,\Pr} (\If^n)
        \le \E \left[U \left( \E \left[  \Upsilon \left(X_T\right)
    -\widehat{g}_{\Pr} \left(\mu,Y_T \right)
    -
    \int_0^T \int_{\R \x \Er} L_{\Pr} \left(t,e \right)\;\Lambda_t(\mathrm{d}z,\mathrm{d}e)\mathrm{d}t \Big|\muh, Y_0
    \right] \right) \right]
    \\
    & = \lim_{j \to \infty} \int_0^1 U \left( \widehat{J}(\gamma^{u,j},\aleph^{u,j}) \right) \mathrm{d}u
    \le U (\widehat{V})
    \end{align*}
    where in the last inequality, we used the fact that $U$ is non--decreasing.
    By taking a bounded Lipschitz $(\gamma,\aleph) \in \Ac$, there exists (see for instance \cite[Theorem 2.5.]{Lacker-strong-2018}) $\Xbb^n=(X^{1,n},\cdots,X^{n,n})$ verifying: $X^{i,n}_0=\iota^i$,
    \begin{align*}
        \mathrm{d}X^{i,n}_t
        =\widehat{b} \left(t, X^{i,n}_t, \mu^n, \aleph(t,X^{i,n}_t), \gamma(t,X^{i,n}_t) \right) \mathrm{d}t + \sigma(t,X^{i,n}_t) \mathrm{d}W^i_t,\;\mu^n_t:=\frac{1}{n} \sum_{i=1}^n \delta_{X^{i,n}_t}
    \end{align*}
    with $\;\mut^n_t:=\frac{1}{n} \sum_{i=1}^n \delta_{\left(X^{i,n}_t, \aleph(t,X^{i,n}_t), \gamma(t,X^{i,n}_t) \right)}$ and
    \begin{align*}
        \lim_{n \to \infty} \Lc\left(\mu^n, \delta_{\tilde \mu^n_t}(\mathrm{d}m) \mathrm{d}t \right)=\Lc\left(\mu,\delta_{\tilde \mu_t}(\mathrm{d}m)\mathrm{d}t \right)\;\mbox{in}\;\Wc_2,\;\mut_t:=\Lc\left(X_t, \aleph(t,X_t), \gamma(t,X_t) \right).
    \end{align*}
    By setting $\If^n:=\left( \aleph(t,X^{i,n}_t), \gamma(t,X^{i,n}_t) \right)_{t \in [0,T]} \in \Sc_n$, we have 
    \begin{align*}
         U \left( \widehat{J}(\gamma,\aleph) \right)= \lim_{n \to \infty}  \widehat{J}_{n,\Pr} (\If^n) \le \liminf_{n \to \infty} \sup_{\If'^n=(\alephbb'^n,\Zbb'^n) \in \Sc_n} \widehat{J}_{n,\Pr} (\If'^n).
    \end{align*}
    By notice that the supremum of $\widehat{V}$ can be taken over a set of bounded Lipschitz admissible controls (see the idea of \Cref{prop:markovian_appr}) and the fact that $\sup_{(\gamma,\aleph) \in \Ac}U \left( \widehat{J}(\gamma,\aleph) \right)=U(\widehat{V})$ (recall that $U$ is non--decreasing), we can deduce the result. 
\end{proof}

\medskip
\paragraph*{From the limit problem to approximate solution of the Principal with $n$--agent problem}
Now, let $(\gamma,\aleph)$ be an admissible control for the problem $\widehat{V}$ s.t.
\begin{align*}
    \E \left[\exp \left(\int_0^T a |\aleph(t,U_t)|^2 + b |\gamma(t,U_t)|^2 \mathrm{d}t \right) \right]< \infty,\;\;\mbox{for each }a,b \ge 0
\end{align*}
with $U_0=\iota$ and $\mathrm{d}U_t=\sigma(t,U_t)\mathrm{d}W_t$. We introduce an $n$--particle system. Let $\Xbb^{\ell,n}:=\left(X^{\ell,1},\cdots,X^{\ell,n} \right)$ be the solution of: $\muh^{\ell,n}_t=\frac{1}{n} \sum_{i=1}^n \delta_{(X^{\ell,i}_t,\aleph^\ell(t,X^{\ell,i}_t),\gamma^\ell(t,X^{\ell,i}_t))},$
\begin{align*}
    \mathrm{d}X^{\ell,i}_t =  \widehat{b}\left(t,X^{\ell,i}_{t}, \mu^{\ell,n},\aleph^\ell(t,X^{\ell,i}_t),\gamma^\ell(t,X^{\ell,i}_t)  \right) \;\mathrm{d}t 
    +
    \sigma(t,X^{\ell,i}_t) \mathrm{d}W^i_t\;\;\mbox{with}\;\;\gamma^\ell(\cdot):=\gamma(\cdot) \wedge \ell,\;\aleph^\ell(\cdot):=\aleph(\cdot) \wedge \ell
\end{align*}
and, we define with $Y^{\aleph,\gamma}_0 \ge R$
\begin{align*}
    &Y^{\ell,n}_\cdot
    :=
    Y^{\aleph,\gamma}_0 -\frac{1}{n} \sum_{i=1}^n\int_0^\cdot H\left(t,X^{\ell,i}_t,\mu^{\ell,n},\aleph^\ell(t,X^{\ell,i}_t),\gamma^\ell(t,X^{\ell,i}_t) \right)\;\mathrm{d}t + \int_0^\cdot \gamma^\ell(t,X^{\ell,i}_t) \sigma(t,X^{\ell,i}_t)^{-1} \mathrm{d}X^{\ell,i}_t,
    \\
    &\;\;\xi^{\ell,n}(\Xbb^{\ell,n}):=g^{-1}\left(\mu^{\ell,n},Y^{\ell,n}_T\right),\;\;\aleph^{\ell,i,n}(t,\Xbb^{\ell,n}):=\aleph^\ell(t,X^{\ell,i}_t)\;\;\mbox{and}\;\;Z^{\ell,i}_t=\frac{\gamma^\ell(t,X^{\ell,i}_t)}{n} .
\end{align*}
We also introduce the control $\alphab^{\ell,n}=(\alpha^{\ell,1,n},\cdots, \alpha^{\ell,n,n})$ by
\begin{align*}
    \alpha^{\ell,i,n}(t,\Xbb^{\ell,n}):= \widehat{\alpha}\left(t,X^{\ell,i}_t,\mu^{\ell,n},\aleph^{\ell,i,n}(t,\Xbb^{\ell,n}_t),\sigma(t,X^{\ell,i}_t)^{-1}\gamma^\ell(t,X^{\ell,i}_t) \right).
\end{align*}
The next Proposition shows that the sequence of contracts we constructed are admissible, the sequence of controls mentioned are optimal for these contracts, and both these sequences converge in Wasserstein to a distribution associated to $(\gamma,\aleph)$. The corollary that comes after just mentions the fact that the sequence of contracts are approximately optimal when $(\gamma,\aleph)$ solves $\widehat{V}$.
\begin{proposition} \label{prop:approx_cong_PA}
    Under {\rm\Cref{assum:main1}}, {\rm \Cref{assump:growth_coercivity}} and {\rm \Cref{assump:growth_coercivity1}},  for each $\ell, n \ge 1$, $\Cf^{\ell,n}=(\xi^{\ell,n},\alephbb^{\ell,n})$ is a contract i.e. $\Cf^{\ell,n} \in \Xi_n$, $\alphab^{\ell,n}$ is an optimal control for the $n$--agent problem i.e. $\alphab^{\ell,n} \in {\rm PE}[\Cf^{\ell,n}]$ and 
    \begin{align*}
        \lim_{\ell \to \infty} \lim_{n \to \infty} \Lc\left( \mu^{\ell,n}, Y^{\ell,n}, \delta_{\hat \mu^{\ell,n}_t}(\mathrm{d}m)\mathrm{d}t \right) = \Lc( \mu, Y^{\gamma, \aleph}, \delta_{\hat \mu_t}(\mathrm{d}m)\mathrm{d}t )\;\mbox{in}\;\Wc_1.
    \end{align*}
    Consequently, $\lim_{\ell \to \infty}\lim_{n \to \infty} \widehat{J}_{n,\Pr} (\alephbb^{\ell,n}, \Zbb^{\ell,n})=U(\widehat{J}\left(\gamma,\aleph) \right)$.
\end{proposition}

\begin{proof}
    For each $\ell, n \ge 1$, we can notice that $\gamma^\ell$ and $\aleph^\ell$ are bounded (by $\ell$), since all the exponential moments of $\Lc(\iota)$ are finite i.e. $\E\left[e^{c|\iota|} \right]< \infty$ for any $c$, by Gronwall lemma, we can show that all the exponential moments of $\sup_{t \le T} |X^{\ell,i}_t|$ are finite. Therefore, $Y^{\ell,n}_T$ has all his exponential moments finite and by using the conditions over $(\aleph,\gamma)$,  $(\xi^{\ell,n},\alephbb^{\ell,n})$ belongs to $\widehat{\Xi}_n$. By \Cref{prop:characterization_agent}, we can verify that $\alphab^{\ell,n}$ is an optimal control. The convergence result is an application of \Cref{prop:first_cong_bounded} and a Propagation of chaos result of \cite[Theorem 2.5.]{Lacker-strong-2018} since for each $\ell \ge 1$, the map $(t,x,\pi) \mapsto \widehat{b}(t,x,\pi,\aleph^\ell(t,x),\gamma^\ell(t,x))$ is bounded and Lipschitz in $m$.
\end{proof}

\begin{corollary} \label{corollary:approx_cong_PA}
    In the context of the previous {\rm \Cref{prop:approx_cong_PA}}, if $(\gamma,\aleph)$ are optimal for $\widehat{V}$ then for each $n, \ell \ge 1$, $\Cf^{\ell,n}=(\xi^{\ell,n},\alephbb^{\ell,n})$ is an $\delta_{\ell,n}$--optimal contract for the problem of the Principal with
    \begin{align*}
        \lim_{\ell \to \infty}\lim_{n \to \infty}\left(\delta_{\ell,n}:=\sup_{\If'^n=(\alephbb'^n,\Zbb'^n) \in \Sc_n} \widehat{J}_{n,\Pr} (\If'^n)- \widehat{J}_{n,\Pr} (\alephbb^{\ell,n}, \Zbb^{\ell,n}) \right)
         =0.
    \end{align*}
\end{corollary}
\begin{proof}
    Notice that, by definition of $\delta_{\ell,n}$, since $(\alephbb^{\ell,n}, \Zbb^{\ell,n}) \in \Sc_n$, we have $\delta_{\ell,n} \ge 0$ and
    $$ 
        \widehat{J}_{n,\Pr} (\alephbb^{\ell,n}, \Zbb^{\ell,n}) \ge \sup_{\If'^n=(\alephbb'^n,\Zbb'^n) \in \Sc_n} \widehat{J}_{n,\Pr} (\If'^n) - \delta_{\ell,n}.
    $$
    By the previous Proposition, we know that $\lim_{\ell \to \infty}\lim_{n \to \infty} \widehat{J}_{n,\Pr} (\alephbb^{\ell,n}, \Zbb^{\ell,n})=U(\widehat{J}\left(\gamma,\aleph) \right)=U(\widehat{V})$. Combining this result with \Cref{prop:conv-principal}, we can deduce that $\lim_{\ell \to \infty}\lim_{n \to \infty}\delta_{\ell,n} = 0.$
\end{proof}

\subsubsection{Proof of \Cref{thm:cong_PA_comp}} 
The proof of \Cref{thm:cong_PA_comp} is just a straightforward application of \Cref{prop:reformulation_principal} and \Cref{prop:conv-principal}. Indeed, by \Cref{prop:reformulation_principal}, the value function of the Principal $V_{n,\Pr}$ is equal to $\widehat{V}_{n,\Pr}$. And, by \Cref{prop:conv-principal}, we have $\lim_{n \to \infty} V_{n,\Pr}=\lim_{n \to \infty} \widehat{V}_{n,\Pr}=U(\widehat{V})$.

\subsubsection{Proof of \Cref{thm:fromLimit}}
For the proof of \Cref{thm:fromLimit}, it is enough to use the equivalence in \Cref{prop:reformulation_principal} i.e. $V_{n,\Pr}=\widehat{V}_{n,\Pr}$  and to apply \Cref{prop:approx_cong_PA} and \Cref{corollary:approx_cong_PA}.

\subsection{Principal--agent problem with McKean--Vlasov dynamics}
\medskip
\paragraph*{The problem of the agent} We present here the result associated to the problem of the agent for the problem of Principal--Agent with McKean--Vlasov dynamics. We recall that the setup has been introduced in  \Cref{section:PA_McKean}.


\medskip
 When the contract belongs to $\widehat{\Xi}$,  the next result shows that the set of optimal controls ${\Ac}^\star(\Cf)$ is actually unique with a specific shape.
\begin{proposition} \label{prop:characterization_control}
     Let {\rm\Cref{assum:main1}},{\rm\Cref{assump:growth_coercivity}} and {\rm\Cref{assump:growth_coercivity1}} be true. Given a contract $\Cf=(\xi,\aleph) \in \widehat{\Xi}$ associated to $\gamma$ and $Y_0$, any optimal control $\alpha$ of the agent's problem satisfies: $\mathrm{d}t \otimes\P$--a.e.
    $$
        \alpha(t,X^\alpha)=\widehat{\alpha}(t,X^\alpha_t,\mu^\alpha, \aleph(t,X^\alpha_t), \sigma(t,X^\alpha_t)^{-1} \gamma(t,X^\alpha_t) ).
    $$
    In addition, $Y_0=\sup_{\alpha \in {\Ac}(\Cf)} {J}^{\Cf}_A(\alpha).$
\end{proposition}

\begin{proof}
     Let $\Cf=(\xi,\aleph) \in \widehat{\Xi}$ be a contract associated to $\gamma$ and $Y_0$, and for any admissible control $\alpha \in {\Ac}(\Cf)$. We can check that 
     \begin{align*}
         &\E \left[ \int_0^T \gamma(t,X^\alpha_t)^2\sigma(t,X^\alpha_t)^{-2} \mathrm{d} \langle X^\alpha \rangle_t  \right]
         =
         \E \left[ \int_0^T \gamma(t,X^\alpha_t)^2\mathrm{d}t  \right] 
         \\
         &=
         \E \left[ \int_0^T L^\alpha_T \gamma(t,U_t)^2\mathrm{d}t  \right]
         \le \E \left[ |L^\alpha_T|^2 \right]^{1/2} \E \left[ \left|\int_0^T \gamma(t,U_t)^2\mathrm{d}t \right|^2 \right]^{1/2}< \infty.
     \end{align*}
     Then,
     \begin{align*}
        & g \left(\mu^\alpha, \overline{\xi}(\mub^\alpha) \right)
        \\
        &=
        Y_0 - \E\left[ \int_0^T H(t,X^\alpha_t,\mu^\alpha, \aleph(t,X^\alpha_t), \gamma(t,X^\alpha_t)) \mathrm{d}t \right] + \E \left[ \int_0^T \gamma(t,X^\alpha_t)\sigma(t,X^\alpha_t)^{-1} \mathrm{d}X^\alpha_t \right]
        \\
        &=Y_0 - \E\left[ \int_0^T H(t,X^\alpha_t,\mu^\alpha, \aleph(t,X^\alpha_t) ,\gamma(t,X^\alpha_t)) \mathrm{d}t \right] 
        \\
        &~~~~~~~~~~~~~~+ \E \left[ \int_0^T \gamma(t,X^\alpha_t)\sigma(t,X^\alpha_t)^{-1} b \left(t, X^\alpha_t, \mu^\alpha, \aleph(t,X^\alpha_t), \alpha(t,X^\alpha) \right) \mathrm{d}t \right].
    \end{align*}
    Consequently,
    \begin{align*}
        {J}^{\Cf}_A(\alpha)
        &=
        \E \left[ Y_0 + \int_0^T h(t,X^\alpha_t,\mu^\alpha, \aleph(t,X^\alpha_t), \gamma(t,X^\alpha_t), \alpha(t,X^\alpha)) -H(t,X^\alpha_t,\mu^\alpha, \aleph(t,X^\alpha_t), \gamma(t,X^\alpha_t))  \mathrm{d}t  \right]
        \\
        &\le \E \left[ Y_0 \right] =  J_{\Cf}(\alpha^\star)
    \end{align*}
    where $\alpha^\star(t,x):=\widehat{\alpha}(t,x_t,\mu,\aleph(t,x_t), \sigma(t,x_t)^{-1} \gamma(t,x_t))$ with $\Lc(X_t)=\mu_t$, $X_0=\iota$ and
    $$
        \mathrm{d}X_t
    =
    \widehat{b}\left(t,X_t,\mu,\aleph(t,X_t), \gamma(t,X) \right) \mathrm{d}t + \sigma(t,X_t) \mathrm{d}W_t.
    $$
    We can check that $\alpha^\star \in {\Ac}(\Cf)$. 
    Therefore, if $\alpha$ is optimal, by definition of the Hamiltonian, we have necessarily $\mathrm{d}t \otimes \mathrm{d}\P$--a.e.
    \begin{align*}
        H(t,X^\alpha_t,\mu^\alpha, \aleph(t,X^\alpha_t), \gamma(t,X^\alpha_t))= h(t,X^\alpha_t,\mu^\alpha, \aleph(t,X^\alpha_t), \gamma(t,X^\alpha_t), \alpha(t,X^\alpha)).
    \end{align*}
    By uniqueness of the maximizer of the Hamiltonian we can deduce the result. 
    
\end{proof}

\subsubsection{Proof of \Cref{prop:PA_Mckean_charac}}

\begin{proof}
    
\medskip
$\boldsymbol{(i)}$ 
$\boldsymbol{\rm Step\;1:\;From\;the\;limit\;to\;the\;n\mbox{--}player\;approximation\;viewpoint}$
Let $\Cf=(\xi,\aleph) \in \Xi$. Let $\alpha$ be an admissible control associated to $\Cf$. We introduce $\widetilde{\alpha}^n(\cdot):=\alpha(\cdot) \wedge n$, $\widetilde{\aleph}^n(\cdot):=\aleph(\cdot) \wedge n$, $\widetilde{\xi}^n(\cdot):= \xi(\cdot) \wedge n$, $g^n(\cdot):= g(\cdot) \wedge n$,
\begin{align*}
    \mathrm{d}X^{i}_t
    =
    b\left(t,X^i_t,\mu^n_t,\widetilde{\aleph}^n(t,X^i_t), \widetilde{\alpha}^n(t,X^i) \right) \mathrm{d}t + \sigma(t,X^{i}_t) \mathrm{d}W^i_t,\;\mut^n_t=\frac{1}{n} \sum_{i=1}^n \delta_{(X^i_t,\;\tilde{\alpha}^n(t,X^i_t), \tilde{\aleph}^n(t,X^i_t) )}\;\mbox{and}\;\mub^n:=\frac{1}{n}\sum_{i=1}^n \delta_{X^i}.
\end{align*}
With the help of \Cref{prop:weak_borel} and the assumptions over the coefficients, we obtain: with $\mut_t=\Lc(X^\alpha_t, \alpha(t,X^\alpha),\aleph(t,X^\alpha_t))$, $\lim_{n \to \infty}  \Lc \left(\delta_{\tilde \mu^n_t}(\mathrm{d}m)\mathrm{d}t \right)=\Lc\left(\delta_{\tilde \mu_t}(\mathrm{d}m)\mathrm{d}t \right)$ in $\Wc_2$, 
Then
\begin{align} \label{eq:cong_AnyControl}
    &\lim_{n \to \infty}\E \left[ \frac{1}{n} \sum_{i=1}^n \int_0^T L\left(t,X^i_t,\mu^n,\widetilde{\aleph}^n(t,X^i_t), \widetilde{\alpha}^n(t,X^i) \right) \mathrm{d}t +  \frac{1}{n} \sum_{i=1}^n \widetilde{\xi}^n(X^i, \mub^n) \right] \nonumber
    \\
    &=
    \E \left[ \int_0^T L\left(t,X^\alpha_t,\mu^\alpha,\aleph(t,X^\alpha_t), \alpha(t,X^\alpha) \right) \mathrm{d}t +  \xi(X^\alpha,\mub^\alpha)\right].
\end{align}

\medskip
For any $(t,x^1,\cdots,x^n) \in [0,T] \x \Cc^n$, let us set $\alpha^{i,n}=\widetilde{\alpha}^n(t,x^i)$,  $\aleph^{i,n}(t,x^1,\cdots,x^n):=\widetilde{\aleph}^n(t,x^i(t))$ and $\xi^n(x^1,\dots,x^n):=\frac{1}{n}\sum_{i=1}^n\widetilde{\xi}^n(x^i,\pi^n)$, $\pi^n:=\frac{1}{n} \sum_{i=1}^n \delta_{x^i}$ and, $\alphab:=(\alpha^{1,n},\cdots,\alpha^{n,n})$ and $\Cf^n:=(\xi^n,\alephbb^n)$.

\medskip
As in the proof of \Cref{prop:characterization_agent}, let us define the processes $(U^i)_{1 \le i \le n}$ as the solution of: $U^i_0=\iota^i$ and $\mathrm{d}U^i_t=\sigma(t,U^i_t) \mathrm{d}W^i_t$. By using \cite[Corollary 6]{briand2008quadratic}, there exists a unique solution $(\widehat{Y}^n, \widehat{\Zbb}^n)$ satisfying: $(\widehat{Y}^n, \widehat{\Zbb}^n) \in \S \x \H^n$, $\widehat{Y}^n_T=\xi^n(\Ubb^n)$,
\begin{align*}
    \mathrm{d} \widehat{Y}^n_t=
    - \frac{1}{n} \sum_{i=1}^n H (t, U^i_t, \nu^n, \aleph^{i,n}(t, \Ubb^n ), n \widehat{Z}^i_t) \mathrm{d}t + \sum_{i=1}^n \widehat{Z}^i_t \mathrm{d}W^i_t\;\mbox{with}\;\nu^n_t=\frac{1}{n} \sum_{i=1}^n \delta_{U^i_t}.
\end{align*}
We define the probability $\P^{\boldsymbol{\alpha}}$ by $\frac{\mathrm{d \P^{\boldsymbol{\alpha}}}}{\mathrm{d}\P}:=L^{n}_T$ with $\phi^i(t,\Ubb^n)
        :=
        \sigma^{-1}(t,U^i_t) b\left(t,U^i_{t}, \nu^n ,\aleph^{i,n}(t,\Ubb^n),\alpha(t,U^i) \right)$ and $\frac{\mathrm{d}L^{n}_t}{L^{n}_t}= \sum_{i=1}^n \phi^i(t,\Ubb^n)\;\mathrm{d}W^i_t$.
    Notice that, since $b$ is bounded,  by Girsanov Theorem, we have that $W^{\alpha,i}_\cdot=W^i_\cdot-\int_0^\cdot \phi^i(s,\Xbb_s) \mathrm{d}s$ is a $\P^{\boldsymbol{\alpha}}$--Brownian motion, 
    and $\Lc \left( \Xbb^{n} \right)=\Lc^{\P^{\boldsymbol{\alpha}}}( \Ubb^n)$. 
    By using similar techniques as in \Cref{prop:characterization_agent}, we have 
\begin{align} \label{eq:equation_M}
    M^{n}_t 
    &=
    \widehat{Y}^n_0 - \frac{1}{n} \sum_{i=1}^n \int_0^t H(s,U^i_s,\nu^n,\aleph^{i,n}(s,\Ubb^n), n \widehat{Z}^i_s) \mathrm{d}s \nonumber 
    \\
    &~~~~~~~~~~+ \frac{1}{n} \sum_{i=1}^n\int_0^t h(s,U^i_s, \nu^n, \aleph^{i,n}(s,\Ubb^n), n \widehat{Z}^i_s, \widetilde{\alpha}^n(s,U^i)) \mathrm{d}s + \sum_{i=1}^n \widehat{Z}^i_s \mathrm{d}W^{\alpha,i}_s
\end{align} 
with 
\begin{align} \label{eq:equation_M2}
    M^{n}_t
    :=
    \widehat{Y}^n_t + \frac{1}{n} \sum_{i=1}^n \int_0^t L\left(s,U^{i}_s,\nu^n,\aleph^{i,n}(s,\Ubb^n),\widetilde{\alpha}^n(s,U^i) \right)\;\mathrm{d}s.
\end{align} 
We can check that $J^{\Cf^n}_{n}(\alphab) = \E^{\P^\alpha}[M^n_T]$ and 
\begin{align} \label{eq:ineq_opti}
    J^{\Cf^n}_{n}(\alphab) \le \E\left[\widehat{Y}^n_0
        \right] = \E^{\P^\alpha}\left[\widehat{Y}^n_0
        \right]
        = \E^{\P^n} \left[ \frac{1}{n} \sum_{i=1}^n\int_0^T \widehat{L}\left(t,U^{i}_t,\nu^n,\aleph^{i,n}(t,\Ubb^n),n \widehat{Z}^i_t \right)\;\mathrm{d}t +\xi^n(\Ubb^n)\right]
\end{align}
where $\frac{\mathrm{d \P^{n}}}{\mathrm{d}\P}:=\widehat{L}^{n}_T\;\;\mbox{with}\;\;\frac{\mathrm{d}\widehat{L}^{n}_t}{\widehat{L}^{n}_t}= \sum_{i=1}^n \sigma^{-1}(t,U^i_t) \widehat{b}\left(t,U^i_{t}, \nu^n ,\aleph^{i,n}(t,\Ubb^n), n \widehat{Z}^i_t \right)\;\mathrm{d}W^i_t$. The last equality is essentially a verification argument by using the fact that $b$ is bounded, so that by Girsanov Theorem again, $\widehat{W}^i_\cdot=W^i_\cdot- \int_0^\cdot \sigma^{-1}(t,U^i_t) \widehat{b}\left(t,U^i_{t}, \nu^n ,\aleph^{i,n}(t,\Ubb^n), n \widehat{Z}^i_t \right) \mathrm{d}t$ is a $\P^n$--Brownian motion and, also $\E^{\P^n} \left[ \sum_{i=1}^n\int_0^T \widehat{Z}^{i,n}_t \mathrm{d}\widehat{W}^{i}_t \right]=0$ and $\E^{\P^{\alphab}} \left[ \sum_{i=1}^n\int_0^T \widehat{Z}^{i,n}_t \mathrm{d}W^{\alpha,i}_t \right]=0$.

\medskip
$\boldsymbol{\rm Step\;2:Checking\;of\;the\;tightness}$
Since $(\widehat{Z}^1,\cdots,\widehat{Z}^n)$ is $(\sigma(\Ubb^n_s:\;s \le t))_{t \in [0,T]}$--adapted, we write $n\widehat{Z}^i_t=C^i(t,\Ubb^n)$. We define $\Xbb=(X^1,\cdots,X^n)$ and $(L^1,\cdots,L^n)$ by 
\begin{align*}
    \mathrm{d}X^i_t= \widehat{b}(t,X^i_t,\eta^n, \aleph^{i,n}(t,\Xbb), C^i(t,\Xbb))\mathrm{d}t + \sigma(t,X^i_t) \mathrm{d}W^i_t,\;\frac{\mathrm{d}L^i_t}{L^i_t}=\sigma(t,X^i_t)^{-1}\widehat{b}(t,X^i_t,\eta^n, \aleph^{i,n}(t,\Xbb), C^i(t,\Xbb))\mathrm{d}W^i_t
\end{align*}
with $X^i_0=\iota^i$, $L^i_0=1$, $\overline{\eta}^n:=\frac{1}{n} \sum_{i=1}^n \delta_{X^i}$ and $\eta^n_t:=\frac{1}{n} \sum_{i=1}^n \delta_{X^i_t}$. Since $b$ is bounded, by Girsanov Theorem, we have $\Lc^{\P^n}(\Ubb^n)=\Lc(\Xbb)$ and $\Lc(X^i)=\Lc^{\Q^i}(U^i)$ where $\Q^i:=L^i_T \mathrm{d}\P$. By using the assumptions over $\xi,$ we obtain,
\begin{align*}
    &\sup_{n \ge 1}\left|\E^{\P^n}\left[ \xi^n(\Ubb^n)\right] \right|
    \le \left|\frac{1}{n} \sum_{i=1}^n\E^{\P}\left[ \xi(X^i,\overline{\eta}^n)\right] \right| \le C \sup_{n \ge 1} \left( 1 + \frac{1}{n} \sum_{i=1}^n\E^{\P}\left[ \psi(X^i)\right] + \frac{1}{n} \sum_{i=1}^n \E^{\P^n}\left[\sup_{t \in [0,T]} |U^i_t|^p \right]  \right) 
    \\
    &\le C \sup_{n \ge 1} \left( 1 + \frac{1}{n} \sum_{i=1}^n\E^{\P}\left[ |L^i_T|^2\right]^{1/2} \E^{\P}\left[ |\psi(U^i)|^2\right]^{1/2} + \frac{1}{n} \sum_{i=1}^n \E^{\P}\left[\sup_{t \in [0,T]} |X^i_t|^p \right]  \right) < \infty
\end{align*}
for a constant $C >0$ independent of $n$. By the assumptions associated to $\widehat{L}$, we have
\begin{align*}
    J^{\Cf^n}_{n}(\alphab) \le \E\left[\widehat{Y}^n_0
        \right] &= \E^{\P^n} \left[ \frac{1}{n} \sum_{i=1}^n\int_0^T \widehat{L}\left(t,U^{i}_t,\nu^n,\aleph^{i,n}(t,\Ubb^n),n \widehat{Z}^i_t \right)\;\mathrm{d}t +\xi^n(\Ubb^n) \right]
        \\
        &\le C\;\E^{\P^n} \left[\frac{1}{n} \sum_{i=1}^n \int_0^T |U^i_t|^p - |\aleph^{i,n}(t,\Ubb^n)|^2 - |n \widehat{Z}^i_t|^2 \mathrm{d}t  \right] + \sup_{n' \ge 1}\E^{\P^{n'}}\left[ \xi^{n'}(\Ubb^{n'})\right],
\end{align*}
for some constant $C > 0$ independent of $n$.  We have necessarily 
\begin{align} \label{eq:bounded_opti}
    \sup_{n \ge 1}  \E^{\P^n} \left[\frac{1}{n} \sum_{i=1}^n \int_0^T |n \widehat{Z}^i_t|^2 \mathrm{d}t  \right] < \infty.
\end{align}
Indeed, if it is not the case, there exists $(n_\ell)_{\ell \ge 1} \subset \N^\star$ s.t. $\lim_{\ell \to \infty} n_\ell=\infty$ and 
$$
    \lim_{\ell \to \infty}\E^{\P^n} \left[\frac{1}{n_\ell} \sum_{i=1}^{n_\ell} \int_0^T |n_\ell \widehat{Z}^i_t|^2 \mathrm{d}t  \right]=\infty.
$$
By the convergence in \eqref{eq:cong_AnyControl}, we will have a contradiction. Also, since \Cref{assump:growth_coercivity1} is satisfied and $b$ is bounded,
\begin{align*}
    &\frac{1}{n} \sum_{i=1}^n\;\;h\left(t,U^{i}_t,\nu^n,\aleph^{i,n}(t,\Ubb^n),n \widehat{Z}^{i,n}_t,\widetilde{\alpha}^n(t,U^i_t) \right)-  H\left(t,U^{i}_t,\nu^n,\aleph^{i,n}(t,\Ubb),n \widehat{Z}^{i,n}_t \right) 
    \\
    &\ge   -  C \frac{1}{n} \sum_{i=1}^n \left( |U^i_t|^p + |\widetilde{\alpha}^n(t,U^i)|^2 + |n \widehat{Z}^i_t|  - |n \widehat{Z}^i_t|^2 \right),
\end{align*} 
for some constant $C >0$ independent of $n$. Notice that
\begin{align*}
    &J^{\Cf^n}_{n}(\alphab)
    \\
    &= \E^{\P^{\alpha}} \left[ \widehat{Y}^n_0 + \frac{1}{n} \sum_{i=1}^n\int_0^T h\left(t,U^{i}_t,\nu^n,\aleph^{i,n}(t,\Ubb^n),n \widehat{Z}^{i,n}_t,\widetilde{\alpha}^n(t,U^i) \right)\; -  H\left(t,U^{i}_t,\nu^n,\aleph^{i,n}(t,\Ubb),n \widehat{Z}^{i,n}_t \right)\;\mathrm{d}t \right] 
    \\
    &\le \E\left[\widehat{Y}^n_0
        \right]. 
\end{align*}
Since $\sup_{n \ge 1}\E\left[\widehat{Y}^n_0
        \right]< \infty$. We deduce that
\begin{align} \label{eq:bounded_alpha}
    \sup_{n \ge 1}  \E^{\P^\alpha} \left[\frac{1}{n} \sum_{i=1}^n \int_0^T |n \widehat{Z}^i_t|^2 \mathrm{d}t  \right] < \infty.
\end{align}

Consequently, by classical techniques (see for instance \cite{djete2019general}), the sequences $(\widehat{\Qr}_n)_{n \ge 1}$ and $(\widehat{\Pr}_n)_{n \ge 1}$  are relatively compact in $\Wc_r$ for any $1 \le r <2$ and bounded in $\Wc_2$, and $({\Qr}_n)_{n \ge 1}$ and $({\Pr}_n)_{n \ge 1}$  are relatively compact for the weak convergence and bounded in $\Wc_1$ where: $\Gamma^i:=\delta_{n \widehat{Z}^i_t}(\mathrm{d}z)\mathrm{d}t$, $\nuh^n:=\frac{1}{n} \sum_{i=1}^n \delta_{(U^i,\Gamma^i)}$,
\begin{align*}
    \widehat{\Qr}_n
    :=
    \frac{1}{n}\sum_{i=1}^n\Lc^{\P^\alpha}\left(U^i, \Gamma^i, \nu^n,\nuh^n\right),\;\Qr_n
    :=
    \frac{1}{n}\sum_{i=1}^n\Lc^{\P^\alpha}\left(U^i, \Gamma^i, \nu^n,\nuh^n,\widehat{Y}^n_0, M^{n}  \right)
\end{align*}
and 
\begin{align*}
    \widehat{\Pr}_n
    :=
    \frac{1}{n}\sum_{i=1}^n\Lc^{\P^n}\left(U^i, \Gamma^i, \nu^n,\nuh^n\right),\;\Pr_n
    :=
    \frac{1}{n}\sum_{i=1}^n\Lc^{\P^n}\left(U^i, \Gamma^i, \nu^n,\nuh^n,\widehat{Y}^n_0  \right).
\end{align*}

$\boldsymbol{\rm Step\;3:\;From\;the\;n\mbox{--}player\;approximation\;viewpoint\;to\;(back\;to)\;the\;limit}$
Let $\Qr$ and ${\Pr}$ be the limits of convergent sub--sequences of $({\Qr}_n)_{n \ge 1}$ and $({\Pr}_n)_{n \ge 1}$ . We use the same notation for the sequence and the convergent sub--sequence. By using some classical weak convergence result for McKean--Vlasov process (see for instance \cite{djete2019general}) and \Cref{prop:weak_borel}, we get ${\Pr}=\Lc(X,\Gamma, \mu, \muh, Y_0)$ where: $X_0=\iota$, $(\muh,Y_0)$ and $W$ are independent, $\Hc_t:=\sigma( \muh \circ (\Xh_{t \wedge \cdot}, \widehat{\Gamma}_{t \wedge \cdot})^{-1}, Y_0)$,
\begin{align*}
    \;\mathrm{d}X_t
    =
    \int_\R \widehat{b}\left(t, X_t, \mu, \aleph(t,X_t), z \right) \Gamma_t(\mathrm{d}z) \mathrm{d}t + \sigma(t,X_t) \mathrm{d}W_t\;\mbox{with}\;\muh=\Lc(X,\Lambda|\Hc_T),\;\mub=\Lc(X|\Hc_T),\;\mu_t=\Lc(X_t|\Hc_t).
\end{align*}
In addition, $\Qr=\Lc(X^\alpha, \Gamma^\alpha, \mu^\alpha, \muh^\alpha, Y^\alpha_0, M^\alpha)$ where $\mu^\alpha=(\Lc(X^\alpha_t))_{t \in [0,T]}$, $\muh^\alpha=\Lc(X^\alpha,\Gamma^\alpha|\Hc^\alpha_T)$ where $\Hc^\alpha_t:=\sigma( \muh \circ (\Xh_{t \wedge \cdot}, \widehat{\Gamma}_{t \wedge \cdot})^{-1}, Y^\alpha_0)$ with $X^\alpha_0$, $(\muh^\alpha,Y^\alpha_0)$ and $W$ independent. Also, $\Lc(X^\alpha| \Hc^\alpha_T)=\Lc(X^\alpha)$. By \Cref{eq:bounded_opti} and \Cref{eq:bounded_alpha}, we can show that 
\begin{align*}
    \lim_{n \to \infty}\E^{\P^{\alpha}} \left[ \left|\sum_{i=1}^n\int_0^T \widehat{Z}^{i,n}_t \mathrm{d}W^{\alpha,i}_t  \right|^2\right] = \lim_{n \to \infty}\E^{\P^{n}} \left[ \left|\sum_{i=1}^n\int_0^T \widehat{Z}^{i,n}_t \mathrm{d}\widehat{W}^{i}_t  \right|^2\right]=0.
\end{align*}
It is easy to verify that $\E[Y_0]=\E[Y^\alpha_0]$.
By using \Cref{eq:equation_M} and \Cref{eq:equation_M2}, and by some localisation techniques as in the proof of \Cref{prop:weak_borel}, since $H(s,x,\pi,e,z) - h(s,x,\pi,e,z,\alpha(s,x)) \ge 0$ by definition of the Hamiltonian, we characterize the limit and obtain: for each $t \le T$,
\begin{align*}
    \E \left[ M^\alpha_t \right]
    & \le 
    \E \bigg[ Y^\alpha_0 -\int_0^t H(s,X^\alpha_s,\mu^\alpha,\aleph(s,X^\alpha_s), z) 
    \\
    &~~~~~~~~~~+\int_0^t h(s,X^\alpha_s, \mu^\alpha, \aleph(s,X^\alpha_s), z, \alpha(s,X^\alpha)) \;\;\Gamma^\alpha_s(\mathrm{d}z) \mathrm{d}s \bigg]
\end{align*}
and also
\begin{align*}
    \E[Y^\alpha_0] \le \E \left[ \xi(X,\mub) + \int_0^T \int_\R \widehat{L} \left(t, X_t,\mu, \aleph(t,X_t), z \right) \Gamma_t(\mathrm{d}z) \mathrm{d}t\right].
\end{align*}
By the convergence of \Cref{eq:cong_AnyControl} and \Cref{eq:ineq_opti}, we get
\begin{align*}      {J}^{\Cf}_A(\alpha)&=\E[M^\alpha_T]= \lim_{n \to \infty}\E \left[ \frac{1}{n} \sum_{i=1}^n \int_0^T L\left(t,X^i_t,\mu^n,\widetilde{\aleph}^n(t,X^i_t), \widetilde{\alpha}^n(t,X^i) \right) \mathrm{d}t +  \frac{1}{n} \sum_{i=1}^n \widetilde{\xi}^n(X^i, \mub^n) \right]
    \\
    &=\lim_{n \to \infty} \E^{\P^{\alphab}}[M^n_T].
\end{align*}
To conclude the first point, we need to verify that $\E[Y_0] = \sup_{\alpha' \in {\Ac}(\Cf)} {J}^\Cf_A(\alpha')$. Since our previous result is true for any $\alpha \in {\Ac}(\Cf)$, we can see that $\sup_{\alpha' \in {\Ac}(\Cf)} {J}^\Cf_A(\alpha') \le \E[Y_0]$.
By \Cref{prop:weak_appr}, there exists a sequence $(\alpha^{\ell,u})_{\ell \ge 1, u \in [0,1]} \subset {\Ac}(\Cf)$ s.t. $u \mapsto {J}^\Cf_A(\alpha^{\ell,u}) $ is Borel measurable and 
$$
    \lim_{\ell \to \infty} \int_0^1 {J}^\Cf_A(\alpha^{\ell,u}) \mathrm{d}u = \E \left[ \xi(X,\mub) + \int_0^T \int_\R \widehat{L} \left(t, X_t,\mu, \aleph(t,X_t), z \right) \Gamma_t(\mathrm{d}z) \mathrm{d}t\right].
$$
Therefore $\sup_{\alpha' \in {\Ac}(\Cf)} {J}_\Cf(\alpha') \ge \E[Y_0]$, hence the equality.

\medskip
$\boldsymbol{(ii)}$ Let us assume that $\alpha$ is optimal. By the previous result, we obtain 
\begin{align*}
    \E \left[ \int_0^T \int_{\R} h\left(t,X^{\alpha}_t,\mu^\alpha,\aleph(t,X^\alpha_t),z,\alpha(t,X^\alpha) \right)\; -  H\left(t,X^{\alpha}_t,\mu^\alpha,\aleph(t,X^\alpha_t),z \right)\; \Gamma^\alpha(t,X^\alpha)(\mathrm{d}z)\mathrm{d}t \right] =0.
\end{align*}
By definition of $h$, we deduce that $\P$--a.e. for $\Gamma^\alpha(t,X^\alpha)(\mathrm{d}z) \mathrm{d}t$--a.e. $(t,z)$
\begin{align*}
    h\left(t,X^{\alpha}_t,\mu,\aleph(t,X^\alpha_t),z,\alpha(t,X^\alpha) \right)= H\left(t,X^{\alpha}_t,\mu,\aleph(t,X^\alpha_t),z \right).
\end{align*}
By uniqueness of the maximizer of the Hamiltonian, we deduce that necessarily there exists a progressively Borel map $\gamma^\alpha:[0,T] \x C([0,T];\R) \to \R$ verifying $\Gamma^\alpha(t,X^\alpha)(\mathrm{d}z)\mathrm{d}t=\delta_{\gamma(t,X^\alpha)}(\mathrm{d}z)\mathrm{d}t$. This leads to obtain that $\alpha(t,X^\alpha)=\widehat{\alpha}\left(t,X^\alpha_t, \mu^\alpha, \aleph(t,X^\alpha_t), \gamma(t,X^\alpha) \right)$. This is enough to deduce the result.
\end{proof}

\paragraph*{Principal--agent problem with McKean--Vlasov dynamics} 
We recall the definition of $V_{\Pr}$ and $\overline{V}_{\Pr}$ are given in \Cref{eq:value_PA_McKean}.

\subsubsection{Proofs of \Cref{thm:PA_McK_reduced} and \Cref{thm:PA_McK_general}}

\begin{proof}
$\boldsymbol{\rm Proof\;of\; \Cref{thm:PA_McK_reduced}}$
     Let $(\aleph,\gamma)$ be bounded control for $\widehat{V}$. We set 
     $$
        \overline{\xi}(\mub):=g^{-1}\left(\mu,\int_{C([0,T];\R)} \xi(x,\mub) \mub(\mathrm{d}x) \right)
    $$
    with: $Y^{\gamma,\aleph}_0 \ge R$, for each $x$ s.t. $\int_0^T \gamma(t,x_t) \sigma(t,x_t)^{-1} \mathrm{d}x_t$ is well--defined,
\begin{align*}
    \xi(x,\mub)
    :=
    Y_0
    -
    \int_0^T  H (t,x, \mu,\aleph(t,x_t), \gamma(t,x_t)) \; \mathrm{d}t + \int_0^T \gamma(t,x_t) \sigma(t,x_t)^{-1} \mathrm{d}x_t
\end{align*} 
and we fix $\xi(x,\mub)=0$ otherwise. We have 
\begin{align*}
    \sup_{(x,\mu)} \frac{|\xi(x,\mub)|}{1 + \mub + \int_0^T |x_t|^p  + |\aleph(t,x_t)|^2 + |\gamma(t,x_t)|^2 \mathrm{d}t + \left| \int_0^T \gamma(t,x_t) \sigma(t,x_t)^{-1} \mathrm{d}x_t \right|} < \infty 
\end{align*}
with $\E\left[ \int_0^T |U_t|^p  + |\aleph(t,U_t)|^2 + |\gamma(t,U_t)|^2 \mathrm{d}t + \left| \int_0^T \gamma(t,U_t) \sigma(t,U_t)^{-1} \mathrm{d}U_t \right|^2 \right] < \infty$. Consequently, $\Cf=(\xi,\aleph)$ belongs to $\Xi$, and given its shape it belongs to $\widehat{\Xi}$. We have then $\alpha \in {\Ac}^\star(\Cf)$ where $\alpha(t,x):=\widehat{\alpha}(t,x,\mu, \aleph(t,x), \sigma(t,x)^{-1} \gamma(t,x))$ with $\mu=(\Lc(X_t)=\mu_t)_{t \in [0,T]}$, $X_0=\iota$ and
    $$
        \mathrm{d}X_t
    =
    \widehat{b}\left(t,X_t,\mu,\aleph(t,X_t), \gamma(t,X_t) \right) \mathrm{d}t + \sigma(t,X_t) \mathrm{d}W_t.
    $$
Also,  $\mub^\alpha=\Lc(X^\alpha)=\mub$ and
\begin{align*}
    \E[\xi(X^\alpha,\mub^\alpha)]
    =
    Y^{\gamma,\aleph}_0
    -
    \E \left[\int_0^T  \widehat{L} (t,X^\alpha_t, \mu^\alpha,\aleph(t,X^\alpha_t), \gamma(t,X^\alpha_t)) \; \mathrm{d}t \right].
\end{align*}
We deduce that $U(\widehat{J}(\aleph,\gamma)) \le {V}_{\Pr}$. Since the supremum of $\widehat{V}$ can be taken over the set of bounded $(\aleph,\gamma)$, therefore $U(\widehat{V}) \le {V}_{\Pr}$ (recall that $U$ is non--decreasing). For any $\Cf=(\xi,\aleph) \in \widehat{\Xi}$ associated to $\gamma$ and $Y_0 \ge R$, and $\alpha \in {\Ac}^\star(\Cf)$, by \Cref{prop:characterization_control}, we deduce that $U(\widehat{V}) \ge {V}_{\Pr}$. 

\medskip
$\boldsymbol{\rm Proof \;of\;\Cref{thm:PA_McK_general}}$
    Now, we have that $b$ is bounded. It is obvious that ${V}_{\Pr} \le \overline{V}_{\Pr}$. Let $\Cf \in \Xi$ and $\alpha \in {\Ac}^\star(\Cf)$. By \Cref{prop:PA_Mckean_charac}, there exists $\gamma \in {\Mc}_0$ s.t. $\mub^\alpha=\Lc(X^\alpha)$,
    \begin{align*}
        \mathrm{d}X^{\alpha}_t
        = \widehat{b}\left(t,X^\alpha_t,\mu^\alpha,\aleph(t,X^\alpha_t), \gamma(t,X^\alpha) \right) \mathrm{d}t + \sigma(t,X^{\alpha}_t) \mathrm{d}W_t
    \end{align*}
    and
    \begin{align*}
        \E \left[\xi(X^\alpha, \mub^\alpha ) \right]
        &=
        \E \left[Y_0
        -
        \int_0^T \widehat{L}\left(t,X^\alpha_t,\mu^\alpha,\aleph(t,X^\alpha_t), \gamma(t,X^\alpha) \right) \mathrm{d}t\right].
    \end{align*}
    By using \Cref{prop:markovian_appr} in order to approximate through Markovian controls, since $Y_0 \ge R$, we find $\overline{V}_{\Pr} \le U(\widehat{V})$. 
    We can then deduce the result.
\end{proof}



\subsubsection{Proof of \Cref{prop:PA_McKean_construction}}

It follows the same idea as in the proof of \Cref{thm:PA_McK_reduced}. If we show that the contract proposed is admissible, we can deduce the result since its optimality becomes straightforward. For the admissibility, we easily check that 
\begin{align*}
    \sup_{(x,\mu)} \frac{|\xi(x,\mub)|}{1 + \|\mub\|_p^p + \int_0^T |x_t|^p  + |\aleph(t,x_t)|^2 + |\gamma(t,x_t)|^2 \mathrm{d}t + \left| \int_0^T \gamma(t,x_t) \sigma(t,x_t)^{-1} \mathrm{d}x_t \right|} < \infty 
\end{align*}
with $\E\left[ \int_0^T |U_t|^p  + |\aleph(t,U_t)|^2 + |\gamma(t,U_t)|^2 \mathrm{d}t + \left| \int_0^T \gamma(t,U_t) \sigma(t,U_t)^{-1} \mathrm{d}U_t \right|^2 \right] < \infty$.

\subsection{Proof of \Cref{prop:multitask}}

Let us first recall that by \Cref{prop:reformulation_principal}, the problem of the Principal is a stochastic control problem i.e. $V_{n,\Pr}=\widehat{V}_{n,\Pr}$. Since there is no instantaneous payment, the notations $\Cf^n=(\xi^n,\aleph^n)$, $\If^n=(\Zbb,\alephbb)$ and $(\gamma,\aleph)$ will be replaced everywhere by just $\xi^n$, $\Zbb$ and $\gamma$. In order to highlight the dependence of $\overline{b}$, all the variables $(V_{n,\Pr},\widehat{V}_{n,\Pr}, \widehat{V}, X^{i,\Zbb}, X^{\gamma}) $ will be denoted $(V^{\overline{b}}_{n,\Pr}, \widehat{V}^{\overline{b}}_{n,\Pr}, \widehat{V}^{\overline{b}}, X^{\overline{b},i,\Zbb}, X^{\overline{b},\gamma})$. When $\overline{b}=\infty$, this will mean that $b(x)=x$.  
We now recall some estimates of SDEs which are essentially successive applications of the Gronwall's Lemma (see for instance \cite[Lemma 4.1]{djete2019general}). By similar ideas to \cite{MFD-2020-closed}, the optimisation problems $\widehat{V}_{n,\Pr}^{\overline{b}}$ and $\widehat{V}^{\overline{b}}$ can be taken over open--loop controls and not closed--loop controls. From now on, all the controls will be assumed of open--loop type. Let $p >1$. Since $|b(x)| \le |x|$ for each $x$, there exists $C>0$ independent of $n$ and $\overline{b}$, for any $\Zbb=(Z^1,\cdots,Z^n)$ and $\gamma$, we have
\begin{align*}
    \frac{1}{n} \sum_{i=1}^n \E \left[ \sup_{t \in [0,T]} |X^{\overline{b},i,\Zbb}_t|^p \right] \le C \left( 1 + \E[|\iota|^p] + \frac{1}{n} \sum_{i=1}^n \E \left[ \int_0^T |n Z^i_t|^p \mathrm{d}t \right] \right)
\end{align*}
and
\begin{align*}
    \E \left[ \sup_{t \in [0,T]} |X^{\overline{b},\gamma}_t|^p \right] \le C \left( 1 + \E[|\iota|^p] + \E \left[ \int_0^T |\gamma_t|^p \mathrm{d}t \right] \right).
\end{align*}
By using the fact that $|b(x)-b(y)| \le |x-y|$, we also find by Gronwall's Lemma (recall that $\Zbb$ and $\gamma$ are open--loop),
\begin{align*}
    \frac{1}{n} \sum_{i=1}^n \E \left[ \sup_{t \in [0,T]} |X^{\overline{b},i,\Zbb}_t-X^{\infty,i,\Zbb}_t| \right] \le C \left( \frac{1}{n} \sum_{i=1}^n \E \left[ \int_0^T |X^{\infty,i,\Zbb}_t| \1_{ \{ |X^{\infty,i,\Zbb}_t| > \overline{b}\} } \mathrm{d}t \right] \right)
\end{align*}
and
\begin{align*}
    \E \left[ \sup_{t \in [0,T]} |X^{\overline{b},\gamma}_t-X^{\infty,\gamma}_t| \right] \le C \left( \E \left[ \int_0^T |X^{\infty,\gamma}_t| \1_{ \{ |X^{\infty,\gamma}_t| > \overline{b}\} } \mathrm{d}t \right] \right).
\end{align*}
By the coercivity of the reward in the optimization problems $\widehat{V}_{n,\Pr}^{\overline{b}}$ and $\widehat{V}^{\overline{b}}$, with similar ideas to \Cref{prop:coercivity} and see also \cite[Proposition 4.17]{djete2019general}, we can find a constant $K$ independent of $n$ and $\overline{b}$ s.t. the problem $\widehat{V}_{n,\Pr}^{\overline{b}}$ can be taken over $\Zbb$ verifying $\frac{1}{n} \sum_{i=1}^n \E \left[ \int_0^T |n Z^i_t|^2 \mathrm{d}t \right] \le K$ and, the problem $\widehat{V}^{\overline{b}}$ can be taken over $\gamma$ verifying $\E \left[ \int_0^T |\gamma_t|^2 \mathrm{d}t \right] \le K$. Therefore, there exists $C > 0$ independent of $n$ and $\overline{b}$ s.t. 
\begin{align*}
    |\widehat{V}_{n,\Pr}^{\overline{b}} - \widehat{V}_{n,\Pr}^{\infty}| + |\widehat{V}^{\overline{b}} - \widehat{V}^{\infty}| \le \sup_{\Zbb} \frac{1}{n} \sum_{i=1}^n \E \left[ \sup_{t \in [0,T]} |X^{\overline{b},i,\Zbb}_t-X^{\infty,i,\Zbb}_t| \right] + \sup_{\gamma} \E \left[ \sup_{t \in [0,T]} |X^{\overline{b},\gamma}_t-X^{\infty,\gamma}_t| \right] \le C\overline{b}^{-1}.
\end{align*}
Recall that $\mathrm{d}X^{\infty,\gamma}_t=(\gamma_t + \overline{\kappa} \E[X^{\infty,\gamma}_t]) \mathrm{d}t + \mathrm{d}W_t$. By simple calculations, $\E[X^{\infty,\gamma}_t]= e^{\overline{\kappa}t} \E[\iota] + \int_0^t e^{\overline{k}(t-s)} \E[\gamma_s] \mathrm{d}s$.  Consequently,
\begin{align*}
    \widehat{V}^{\infty} = \sup_{\gamma} \E \left[ -R + e^{\overline{\kappa}T} \iota + \int_0^T e^{\overline{k}(T-t)} \gamma_t - \frac{1}{2} \int_0^T |\gamma_t|^2 \mathrm{d}t \right].
\end{align*}
The supremum is then reached at $\widehat{\gamma}_t=e^{\overline{k}(T-t)}$. Then, $|\widehat{V}^{\overline{b}} - \widehat{J}^{\overline{b}}(\widehat{\gamma})| \le |\widehat{V}^{\overline{b}} - \widehat{V}^{\infty}| + |\widehat{J}^{\infty}(\widehat{\gamma}) - \widehat{J}^{\overline{b}}(\widehat{\gamma})| \le C \overline{b}^{-1}.$ With $\widetilde{V}_{n,\Pr}^{\infty}:=\sup_{\Zbb} \E \left[ U \left( \frac{1}{n} \sum_{i=1}^n X^{\infty,i,\Zbb}_T - R - \int_0^T |n Z^i_t|^2 \mathrm{d}t \right) \right]$, remember that we are taking the supremum over $\frac{1}{n} \sum_{i=1}^n \E \left[ \int_0^T |n Z^i_t|^2 \mathrm{d}t \right] \le K$, by using the fact that $U$ is Lipschitz, we have $|\widehat{V}_{n,\Pr}^{\infty}-\widetilde{V}_{n,\Pr}^{\infty}| \le C n^{-1/2}$. We recall that $\mathrm{d}X^{\infty,i,\Zbb}_t=(nZ^i_t +  \frac{\overline{\kappa}}{n} \sum_{j=1}^n X^{\infty,j,\Zbb}_t) \mathrm{d}t + \mathrm{d}W_t$. By simple calculations again, $\frac{1}{n} \sum_{j=1}^n X^{\infty,j,\Zbb}_t= e^{\overline{\kappa}t} \frac{1}{n} \sum_{j=1}^n \iota^j + \int_0^t e^{\overline{\kappa}(t-s)} \frac{1}{n} \sum_{j=1}^n n Z^{j}_s \mathrm{d}s + \int_0^t e^{\overline{\kappa}(t-s)} \frac{1}{n} \sum_{j=1}^n \mathrm{d}W^j_s.$
We obtain $\widetilde{V}_{n,\Pr}^{\infty}:=\sup_{\Zbb} \widetilde{J}^{\infty}_{n,\Pr}(\Zbb)$ where
\begin{align*}
    \widetilde{J}^{\infty}_{n,\Pr}(\Zbb):= \E \left[ U \left( e^{\overline{\kappa}T} \frac{1}{n} \sum_{j=1}^n \iota^j + \int_0^T e^{\overline{\kappa}(T-t)} \frac{1}{n} \sum_{j=1}^n n Z^{j}_t \mathrm{d}t + \int_0^T e^{\overline{\kappa}(T-t)} \frac{1}{n} \sum_{j=1}^n \mathrm{d}W^j_t - R - \int_0^T \frac{1}{n} \sum_{j=1}^n|n Z^i_t|^2 \mathrm{d}t \right) \right].
\end{align*}
Since $U$ is non--decreasing, we compute its supremum and find $\widehat{Z}^i_t=e^{\overline{\kappa}(T-t)}$. Then, by combining all the previous results, 
$$
    |V_{n,\Pr}^{\overline{b}}- \widehat{J}_{n,\Pr}^{\overline{b}}(\widehat{\Zbb})| \le |V_{n,\Pr}^{\overline{b}}-\widehat{V}_{n,\Pr}^{\infty}| + | \widehat{V}_{n,\Pr}^{\infty}- \widehat{J}_{n,\Pr}^{\overline{b}}(\widehat{\Zbb})|  \le |V_{n,\Pr}^{\overline{b}}-\widehat{V}_{n,\Pr}^{\infty}| + | \widehat{V}_{n,\Pr}^{\infty}-\widetilde{V}_{n,\Pr}^{\infty} |+ |\widetilde{J}^{\infty}_{n,\Pr}(\widehat{\Zbb})- \widehat{J}_{n,\Pr}^{\overline{b}}(\widehat{\Zbb})| \le C(n^{-1/2} + b^{-1}).
$$
The calculations of the contract and the utility of the principal are straightforward. We can deduce the proof of the proposition.



\begin{appendix}

\section{Some technical results}

\medskip
Let $(\delta_n)_{n \ge 1}$ be a non--negative sequence s.t. $\lim_{n \to \infty} \delta_n=0$. We consider the sequence $(\If^n=\alephbb^n,\Zbb^n)_{n \ge 1}$ s.t. $\If^n \in \Sc_n$ satisfying: for each $n \ge 1$,
\begin{align*}
    \widehat{J}_{n,\Pr} (\If^n) \ge \widehat{V}_{n,\Pr} - \delta_n.
\end{align*}
\begin{proposition} \label{prop:coercivity}
    Under {\rm\Cref{assum:main1}} and {\rm \Cref{assump:growth_coercivity1}}, the sequence $(\If^n)_{n \ge 1}$ previously considered satisfied
    \begin{align*}
        \sup_{n \ge 1} \E \left[ |Y^n_0| + \frac{1}{n} \sum_{i=1}^n \int_0^T |nZ^{i,n}_t|^{2} + |\aleph^{i,n}(t,\Xbb^n)|^{2} \mathrm{d}t \right] < \infty.
    \end{align*}
\end{proposition}
\begin{proof}
    We need to use the coercivity of the coefficients to show this result. For any $n \ge 1$,
    \begin{align*}
        &\widehat{J}_{n,\Pr} (\If^n)
        \le 
        \E \left[ U \left( \frac{C}{n} \sum_{i=1}^n\;\; |X^i_T|^p - Y^n_T - \int_0^T |\aleph^{i,n}(t,\Xbb^n)|^2 \mathrm{d}t  \right)
        \right]
        \\
        &\le 
         U \left( \E \left[ \frac{C}{n} \sum_{i=1}^n\;\; |X^i_T|^p - R + \int_0^T \widehat{L}(t,X^{i,n}_t, \mu^n_t, \aleph^{i,n}(t,\Xbb^n),Z^{i,n}_t) \mathrm{d}t - \int_0^T C_{\Pr} |\aleph^{i,n}(t,\Xbb^n)|^2 \mathrm{d}t \right]  \right)
         \\
        &\le 
         U \bigg( \E \bigg[ \frac{C}{n} \sum_{i=1}^n\;\; \sup_{t \in [0,T]}|X^i_t|^p - R + \int_0^T C\left(1 + |\aleph^{i,n}(t,\Xbb^n)|^p + |n\;Z^{i,n}_t|^p \right) - C_{\Pr} \left( |\aleph^{i,n}(t,\Xbb^n)|^2 + |n\;Z^{i,n}_t|^2 \right)  \mathrm{d}t 
         \\
         &~~~~~~~~~~~~~~~~~~~~~~~~~~~~~~~~~~~~~~~~~~~~- \int_0^T |\aleph^{i,n}(t,\Xbb^n)|^2 \mathrm{d}t \bigg]  \bigg)
         \\
        &\le 
         U \left( \E \left[ -R + \frac{1}{n} \sum_{i=1}^n \int_0^T C \left( 1 + |\aleph^{i,n}(t,\Xbb^n)|^p + |n\;Z^{i,n}_t|^p \right) - C_{\Pr} \left( |n\;Z^{i,n}_t|^2 + |\aleph^{i,n}(t,\Xbb^n)|^2 \right)\;\;\mathrm{d}t  \right]  \right)
    \end{align*}
    where we used the fact that: with classical gronwall techniques for SDE (see \cite{djete2019general} for instance)
    \begin{align*}
        \frac{1}{n} \sum_{i=1}^n\E\left[ \sup_{t \in [0,T]}|X^i_t|^p \right] \le K \left( 1 +  \frac{1}{n} \sum_{i=1}^n\E \left[\int_0^T |\aleph^{i,n}(t,\Xbb^n)|^p + |n\;Z^{i,n}_t|^p \mathrm{d}t \right]  \right)\;\mbox{for some constant }K>0.
    \end{align*}
    By taking the control $\widetilde{\If}^{n}=(\widetilde{\alephbb}^n,\widetilde{\Zbb}^n)$ where $\widetilde{Z}^{i,n}_t=0$ and $\widetilde{\aleph}^{i,n}(t,\widetilde{\Xbb}^n)=e$ for some $e \in \Er$ s.t. $Y^{n,\widetilde{\If}^{n}}_0 \ge R$, we get that
    \begin{align*}
        |\widehat{J}_{n,\Pr} (\widetilde{\If}^n)| \le \widehat{C},\;\;\mbox{for some constant }\widehat{C} >0\;\mbox{ independent of }\;n.
    \end{align*}
    Consequently, $\inf_{n \ge 1} \widehat{V}_{n,\Pr} > - \infty$. By using the estimates, we get
    \begin{align*}
       &\inf_{n \ge 1}\widehat{V}_{n,\Pr} - \sup_{n \ge 1} \delta_n 
       \\
       &\le \inf_{n \ge 1} U \left( \E \left[ -R + \frac{1}{n} \sum_{i=1}^n \int_0^T C \left( 1 + |\aleph^{i,n}(t,\Xbb^n)|^p + |n\;Z^{i,n}_t|^p \right) - C_{\Pr} \left( |n\;Z^{i,n}_t|^2 + |\aleph^{i,n}(t,\Xbb^n)|^2 \right)\;\;\mathrm{d}t  \right]  \right).
    \end{align*}
    Since $1 \le p <2$, $\lim_{x \to - \infty} U(x)=- \infty$ and $\inf_{n \ge 1} \widehat{V}_{n,\Pr}> - \infty$, we find that necessary we have that 
    \begin{align*}
        \sup_{n \ge 1} \E \left[ \frac{1}{n} \sum_{i=1}^n \int_0^T |n\;Z^{i,n}_t|^2 + |\aleph^{i,n}(t,\Xbb^n)|^2 \mathrm{d}t  \right] < \infty.
    \end{align*}
    By similar techniques, we also show that $\sup_{n \ge 1} \E [|Y^n_0|] < \infty$.

\end{proof}

\medskip
Let $(\If^n)_{n \ge 1}$ be a sequence s.t. $\If^n \in \Sc_n$ and $\sup_{n \ge 1} \E \left[ |Y^n_0| + \frac{1}{n} \sum_{i=1}^n \int_0^T |n\;Z^{i,n}_t|^{2} + |\aleph^{i,n}(t,\Xbb^n)|^{2} \mathrm{d}t \right] < \infty$. We set the sequence $(\Qr^n)_{n \ge 1}$ by
\begin{align*}
    \Qr^n
    :=
    \frac{1}{n} \sum_{i=1}^n \P \circ \left( X^{i,\If^n}, W^i,\Lambda^i, \muh^n \right)^{-1},\;\Lambda^i:=\delta_{ (n Z^{i,n}_t, \aleph^{i,n}(t,\Xbb^{n,\If^n}))}(\mathrm{d}z,\mathrm{d}e)\mathrm{d}t,\;\muh^n:=\frac{1}{n}\sum_{i=1}^n \delta_{(X^{i,\If^n},W^i,\Lambda^i)}.
\end{align*}
In addition, we consider the distribution $\widehat{\Qr}_n:=\frac{1}{n} \sum_{i=1}^n \P \circ \left( X^{i,\If^n}, W^i,\Lambda^i, \muh^n, Y^n_0 \right)^{-1}$.

\begin{proposition}\label{prop:relaxed_appr}{\rm \cite[Proposition 2.17]{djete2019general}}
    The sequence $(\widehat{\Qr}^n)_{n \ge 1}$ is relatively compact for the weak topology and bounded in $\Wc_1$. The sequence $(\Qr^n)_{n \ge 1}$ is relatively compact in $\Wc_r$ with $1 \le r<2$ and bounded in $\Wc_2$. Any point of accumulation $(\widehat{\Qr},\Qr)$ of $(\widehat{\Qr}^n,\Qr^n)_{n \ge 1}$ satisfies: $\widehat{\Qr}=\P \circ (X,W,\Lambda,\muh, Y_0)^{-1}$, $\Qr=\P \circ (X,W,\Lambda,\muh)^{-1}$ where the variables $X_0,W$ and $(\muh, Y_0)$ are independent, $W$ is a Brownian motion, $X$ is an $\F$--adapted continuous process solving
    \begin{align*}
        \mathrm{d}X_t
        =\int_{\R \x \Er}\widehat{b} \left( t, X_t, \mu, e,z \right) \Lambda_t(\mathrm{d}e,\mathrm{d}z) \mathrm{d}t + \sigma(t,X_t)\mathrm{d}W_t,\;\muh=\Lc(X,W,\Lambda|\muh,Y_0),
    \end{align*}
    and $\mu=(\mu_t=\Lc(X_t|\muh,Y_0))_{t \in [0,T]}.$
\end{proposition}

\begin{proposition} \label{prop:weak_appr}
    Let $\widehat{\Qr}=\P \circ \left(X,W,\Lambda,\muh \right)^{-1}$ be a distribution satisfying the similar properties of any points of accumulation of $(\widehat{\Qr}^n)_{n \ge 1}$. There exists a $[0,1]$--valued uniform random variable $U$ independent of $(X_0,W)$ and a sequence of bounded piece--wise constant $\Er \x \R$--valued $(\sigma(X^\ell_s, Y_0, U:\;s \le t))_{t \in [0,T]}$--predictable processes $(\aleph^\ell, \gamma^\ell)_{\ell \ge 1}$ s.t. if $X^\ell$ satisfies: $X^\ell_0=X_0$,
    \begin{align*}
        \mathrm{d}X^\ell_t
        =
        \widehat{b} \left(t, X^\ell_t, \mu^\ell, \aleph^\ell_t, \gamma^\ell_t \right) \mathrm{d}t + \sigma(t,X^\ell_t) \mathrm{d}W_t,\;\Lambda^\ell:=\delta_{(\gamma^\ell_t,\aleph^\ell_t)}(\mathrm{d}z,\mathrm{d}e)\mathrm{d}t,\;\muh^\ell=\Lc(X^\ell,W,\Lambda^\ell|U,Y_0), \mu^\ell=(\Lc(X^\ell_t|U,Y_0))_{t \in [0,T]},
    \end{align*}
    we have
    \begin{align*}
        \lim_{\ell \to \infty} \P \circ \left(X^\ell, W, \Lambda^\ell, \muh^\ell \right)^{-1}= \Qr\mbox{ in }\Wc_2.
    \end{align*}

\end{proposition}

\begin{proof}
    This is essentially an application of \cite[Proposition 4.10.]{djete2019general} combined with \cite[Proposition 4.5.]{djete2019general}. The only difference is the mesurability of $(\aleph^\ell,\gamma^\ell)$. In \cite[Proposition 4.5.]{djete2019general}, the process $(\aleph^\ell,\gamma^\ell)$ is bounded piece--wise constant $(\sigma(W_s, Y_0, U:\;s \le t))_{t \in [0,T]}$--predictable ($Y_0$ can be used as the common noise for the analogy with \cite[Proposition 4.5.]{djete2019general}). Therefore, there exists $0< t^1< \cdots<t^\ell=T$ s.t. the process $(\aleph^\ell,\gamma^\ell)$ verifies $(\aleph^\ell_t,\gamma^\ell_t)=(\aleph^\ell_{t^{k}},\gamma^\ell_{t^k})=\left(\aleph^\ell(t^k, W_{t^k \wedge \cdot},Y_0,U),\gamma^\ell_{t^k} (t^k, W_{t^k \wedge \cdot},Y_0,U) \right)$ for each $t^k \le t < t^{k+1}$. By using the recursivity of the process, we can show that $W$ is $(\sigma(X^\ell_s, Y_0,U:s \le t))_{t \in [0,T]}$--adapted. Indeed, it is enough to notice that: when $t^k \le t < t^{k+1}$,
    \begin{align*}
        W_{t} - W_{t^k}= \int_{t^k}^t \sigma(s,X^\ell_s)^{-1} \mathrm{d}X^\ell_s -  \int_{t^k}^t \sigma(s,X^\ell_s)^{-1} \widehat{b}\left(s,X^\ell_s, \mu^\ell, \aleph^\ell(t^k, W_{t^k \wedge \cdot},Y_0,U),\gamma^\ell_{t^k} (t^k, W_{t^k \wedge \cdot},Y_0,U) \right) \mathrm{d}s.
    \end{align*}
\end{proof}

\begin{proposition} \label{prop:markovian_appr}
    In the context of the previous proposition, for each $\ell \ge 1$, there exists a sequence of bounded Lipschitz maps $\left((\aleph^{\ell,j}, \gamma^{\ell,j}):[0,T] \x \R\x \R \x [0,1] \to \Er \x \R \right)_{j \ge 1}$ s.t. if $X^{\ell,j}$ satisfies
    \begin{align*}
        \mathrm{d}X^{\ell,j}_t
        =
        \widehat{b} \left(t, X^{\ell,j}_t, \mu^\ell, \aleph^{\ell,j}(t,X^{\ell,j}_t,Y_0,U), \gamma^{\ell,j}(t,X^{\ell,j}_t,Y_0,U) \right) \mathrm{d}t + \sigma(t,X^{\ell,j}_t) \mathrm{d}W_t,\;\mu^{\ell,j}_t=\Lc(X^{\ell,j}_t|Y_0,U),
    \end{align*}
    $\mut^{\ell,j}_t=\Lc\left(X^{\ell,j}_t, \aleph^{\ell,j}(t,X^{\ell,j}_t,U), \gamma^{\ell,j}(t,X^{\ell,j}_t,U)|Y_0,U \right)$, we have 
    \begin{align*}
        \lim_{j \to \infty} \P \circ \left( \mu^{\ell,j}, \delta_{\tilde\mu^{\ell,j}_t}(\mathrm{d}m)\mathrm{d}t \right)^{-1}= \P \circ \left( \mu^\ell, \delta_{\tilde \mu^\ell_t} (\mathrm{d}m)\mathrm{d}t  \right)^{-1}\;\mbox{in}\;\Wc_2.
    \end{align*}
    
\end{proposition}

\begin{proof}
    This is just an application of \cite[Proposition A.5.]{MFD-2020-closed}. Indeed, since the control $(\aleph^\ell,\gamma^\ell)$ is bounded for each $\ell \ge 1$, we have $\sup_{(t,\om)}|\aleph^\ell_t| + |\gamma^\ell_t| \le C^\ell$ for some constant $C^\ell >0$. The drift $(t,x,m,e,z) \mapsto \widehat{b}(t,x,m,h_\ell(e),h_\ell(z) )$ is then bounded, continuous in $(x,m,e,z)$ for each $t$ and Lipschitz in $(x,m)$ uniformly in $(t,e,z)$, where $h_\ell(x)=-C^\ell \vee ( C^\ell \wedge x)$. We can then apply \cite[Proposition A.5.]{MFD-2020-closed}.
\end{proof}

\medskip
Let $(\aleph,\gamma) \in \Ac$ be admissible control s.t.
\begin{align*}
    \E \left[\exp \left(\int_0^T a |\aleph(t,U_t)|^2 + b |\gamma(t,U_t)|^2 \mathrm{d}t \right) \right]< \infty,\;\;\mbox{for each }a,b \ge 0.
\end{align*}

We define $\aleph^\ell(\cdot):=\aleph(\cdot) \wedge \ell$ and $\gamma^\ell(\cdot):=\gamma(\cdot) \wedge \ell$. It is straightforward that $(\aleph^\ell,\gamma^\ell) \in \Ac$ is admissible. We set $X^\ell:=X^{\aleph^\ell,\gamma^\ell}$ and $X:=X^{\aleph,\gamma}$ ($X$ is well--posed by a straightforward extension of \cite[Theorem 2.3, Theorem 2.4]{Lacker-strong-2018}). We consider a progressively Borel map $\Phi:[0,T] \x \R \x C([0,T];\Pc_p(\R)) \x \Er \x \R \to \R $ verifying
\begin{align*}
    A:=\sup_{(t,x,\pi,e,z)} \frac{|\Phi(t,x,\pi,e,z)|}{1 + |x|^p + \sup_{s \le t}\|\pi(s)\|^p_p + |e|^2 + |z|^2} < \infty
\end{align*}
and continuous in $(\pi,e,z)$ for any $(t,x)$.
\begin{proposition} \label{prop:first_cong_bounded}
    For $\mub^\ell:=\Lc(X^{\ell})$ and $\mub=\Lc(X)$, we have $\lim_{\ell \to \infty} \Wc_r(\mub^\ell,\mub)=0$ for any $r \ge 1$ , and
    \begin{align*}
        \lim_{\ell \to \infty} \E \left[ \int_0^T \Phi(t,X^\ell_t, \mu^\ell, \aleph^\ell(t,X^\ell_t), \gamma^\ell(t,X^\ell_t)) \mathrm{d}t \right] = \E \left[ \int_0^T \Phi(t,X_t, \mu, \aleph(t,X_t), \gamma(t,X_t)) \mathrm{d}t \right].
    \end{align*}
\end{proposition}

\begin{proof}
    Notice that, for any $\ell \ge 1$, $\Lc^{\P^\ell} (U)=\Lc(X^\ell)=\mub^\ell$ where $\mathrm{d}\P^\ell:=L^{\mu^\ell,\aleph^\ell,\gamma^\ell}_T \mathrm{d}\P$, and for each $(\aleph,\gamma)$, $L^{\mu,\aleph,\gamma}_0=1$, $\mathrm{d}L^{\mu,\aleph,\gamma}_t=L^{\mu,\aleph,\gamma}_t \sigma(t,U)^{-1}\widehat{b}(t,U_t,\mu,\aleph(t,U_t), \gamma(t,U_t)) \mathrm{d}W_t$. For any $r \ge 1$, there exists $C >0$ and $q >1$ s.t.
    \begin{align*}
        \sup_{\ell \ge 1}\E \left[|L^{\mu^\ell, \aleph^\ell,\gamma^\ell}_T|^r \right] \le C \left( 1 + \E \left[\exp \left(\int_0^T q |\aleph(t,U_t)|^2 + q |\gamma(t,U_t)|^2 \mathrm{d}t \right) \right] \right) < \infty
    \end{align*}
    and
    \begin{align*}
        \sup_{\ell \ge 1}\E \left[ \sup_{t \in [0,T]} |X^\ell_t|^r \right]= \sup_{\ell \ge 1} \E \left[ L^{\mu^\ell,\aleph^\ell,\gamma^\ell}_T\sup_{t \in [0,T]} |U_t|^r \right] \le \sup_{\ell \ge 1}\E[|L^{\mu^\ell,\aleph^\ell,\gamma^\ell}_T|^r]^{1/r} \E \left[ \sup_{t \in [0,T]} |U_t|^{r r'} \right]^{1/r'} < \infty
    \end{align*}
    where $r'$ is the conjugate of $r$ i.e. $1/r+1/r'=1$. Let $\Qr^\ell:= \Lc\left( L^{\mu^\ell, \aleph^\ell,\gamma^\ell}, U \right)$. 
    It is easy to check that $(\mub^\ell)_{\ell \ge 1}$ and $(\Qr^\ell)_{\ell \ge 1}$ are relatively compact in $\Wc_r$ for any $r \ge 1$. Let $\mut$ and $\widetilde{\Qr}=\Lc(L, U)$ be the limit of  convergent sub--sequence of $(\mub^\ell)_{\ell \ge 1}$ and $(\Qr^\ell)_{\ell \ge 1}$. We use the same notation for the sequence and the sub--sequence. By noticing that $\lim_{\ell \to \infty}\E \left[ \int_0^T |\aleph^\ell(t,U_t)-\aleph(t,U_t)| \mathrm{d}t \right]=0$ and by using a martingale problem combined with the stable convergence of \citeauthor*{jacod1981type} \cite{jacod1981type}, it is straightforward to check that $\Lc(L,U)=\Lc(L^{\mut,\aleph,\gamma},U)$ and $\Lc^{\widetilde{\P}}(U)=\mut$ where $\mathrm{d}\widetilde{\P}:=L^{\tilde \mu,\aleph,\gamma}_T\mathrm{d}\P$. We deduce by uniqueness that $\mut=\mub$. Then 
    \begin{align*}
        &\lim_{\ell \to \infty}\E \left[ \int_0^T \Phi(t,X^\ell_t, \mu^\ell, \aleph^\ell(t,X^\ell_t), \gamma^\ell(t,X^\ell_t)) \mathrm{d}t \right]
        =\lim_{\ell \to \infty}\E \left[ L^{\mu^\ell,\aleph^\ell,\gamma^\ell}_T \int_0^T \Phi(t,U_t, \mu^\ell, \aleph^\ell(t,U_t), \gamma^\ell(t,U_t)) \mathrm{d}t \right]
        \\
        &=\lim_{\ell \to \infty}\E \left[ L^{\mu^\ell,\aleph^\ell,\gamma^\ell}_T \int_0^T \Phi(t,U_t, \mu, \aleph(t,U_t), \gamma(t,U_t)) \mathrm{d}t \right]
        =\E \left[ L^{\mu,\aleph,\gamma}_T \int_0^T \Phi(t,U_t, \mu, \aleph(t,U_t), \gamma(t,U_t)) \mathrm{d}t \right]\\
        &=\E \left[ \int_0^T \Phi(t,X_t, \mu, \aleph(t,X_t), \gamma(t,X_t)) \mathrm{d}t \right].
    \end{align*}
    This is enough to deduce the proposition.
\end{proof}
We stay in the context of \Cref{prop:first_cong_bounded}. For each $\ell \ge 1$, we consider the process $\Xbb=(X^{1,n},\cdots,X^{n,n})$ verifying: $X^{i,n}_0=\iota^i$,
\begin{align*}
    \mathrm{d}X^{i,n}_t=\widehat{b}\left(t, X^{i,n}_t, \mu^n, \aleph^\ell(t,X^{i,n}_t), \gamma^\ell(t,X^{i,n}_t) \right) \mathrm{d}t + \sigma(t,X^{i,n}_t)\mathrm{d}W^i_t\;\mbox{with}\;\mu^n_t:=\frac{1}{n}\sum_{i=1}^n \delta_{X^{i,n}_t}.
\end{align*}
We set $\mut^n_t:=\frac{1}{n}\sum_{i=1}^n \delta_{\left( X^{i,n}_t, \aleph^\ell(t,X^{i,n}_t), \gamma^\ell(t,X^{i,n}_t) \right)}$ and $\mut^\ell_t:= \Lc\left( X^{\ell}_t, \aleph^\ell(t,X^{\ell}_t), \gamma^\ell(t,X^{\ell}_t) \right)$.
\begin{proposition} \label{prop:second_cong_bounded}
    For each $\ell \ge 1$, $\lim_{n \to \infty} \delta_{\tilde \mu^n_t}(\mathrm{d}m) \mathrm{d}t = \delta_{\tilde\mu^\ell_t}(\mathrm{d}m)\mathrm{d}t\;\mbox{in}\;\Wc_2.$
\end{proposition}
\begin{proof}
    For each $\ell \ge 1$, notice that the map $(t,x,m) \mapsto \widehat{b}(t,x,m,\aleph^\ell(t,x),\gamma^\ell(t,x))$ is bounded and Lipschitz in $m$ for $\Wc_p$. Therefore, we just need to apply \cite[Theorem 2.5.]{Lacker-strong-2018} to deduce the result.
\end{proof}

\medskip
Let $(X^{i,n})_{1 \le i \le n, n \ge 1}$ be a sequence of processes satisfying: $X^{i,n}_0=\iota^i$,
\begin{align*}
    \mathrm{d}X^{i,n}_t=b^{i,n}(t,\Xbb^n)\mathrm{d}t + \sigma(t,X^{i,n}_t) \mathrm{d}W^i_t,\;\Xbb^n:=(X^{1,n},\;\dots,X^{n,n}),\;\mu^n:=\frac{1}{n}\sum_{i=1}^n \delta_{X^{i,n}}.
\end{align*}
where the map $b^{i,n}: [0,T] \x C([0,T];\R)^n \to \R$ is a progressively Borel measurable satisfying the boundness condition $\sup_{n \ge 1}\sup_{i \le n}\sup_{(t,(x^1,\cdots,x^n))} |b^{i,n}(t,x^1,\cdots,x^n)| < \infty$. Also, we take a sequence $(\Lambda^{i,n})_{n \ge 1, i \le n}$ s.t. $\Lambda^{i,n}: C([0,T];\R)^n \to \M(\Er \x \R)$ is Borel measurable.  We consider a Borel map $H:C([0,T]\;\R) \to \R$ and  a non--negative Borel map $F:C([0,T];\R)\x \M(\Er \x \R) \x \Pc_p(C([0,T];\R)) \to \R_+$  s.t. for each $x \in C([0,T];\R)$, $(m,\lambda) \mapsto F(x, \lambda,m)$ is continuous, and
\begin{align*}
    \sup_{x,m}\frac{F(x,\lambda,m)}{1 + \psi(x) + \|m\|^p_p + \|\lambda\|^2_2} + \frac{|H(x)|}{1 + \psi(x)} < \infty
\end{align*}
where $\psi$ is a non--negative Borel map satisfying: $\E[\psi(U)^r]< \infty$ with $U_0=\iota$ and $\mathrm{d}U_t=\sigma(t,U_t)\mathrm{d}W_t$, for some $r > 1$. We define the distribution
\begin{align*}
    \Rr^n:= \frac{1}{n} \sum_{i=1}^n\Lc\left( F^{i,n}(\Xbb^n), H(X^{i,n}), H(X^{i,n}) \wedge n, \overline{F}^n(\Xbb^n)  \right)
\end{align*}
where $F^{i,n}(\Xbb^n):=F(X^{i,n}, \Lambda^{i,n}(\Xbb^n), \mu^n)$ and $\overline{F}^n(\Xbb^n):=\frac{1}{n} \sum_{i=1}^n F(X^{i,n}, \Lambda^{i,n}(\Xbb^n), \mu^n).$

\begin{proposition}\label{prop:weak_borel}
    If $\sup_{n \ge 1} \frac{1}{n} \sum_{i=1}^n \E \left[ \| \Lambda^{i,n}(\Xbb^n))\|^2_2 \right]< \infty$ and $\frac{1}{n} \sum_{i=1}^n\left(\Lc(X^{i,n},\Lambda^{i,n}(\Xbb^n),\mu^n) \right)_{n \ge 1}$ converges to $\Lc\left(X, \Lambda, \mu \right)$ in $\Wc_{r}$ with $p \le r < 2$ for some random variables $(X,\Lambda,\mu)$. Then, any weakly convergent sub--sequence of $(\Rr^n)_{n \ge 1}$ converges to $\Lc\left(F(X,\Lambda,\mu),H(X), H(X),  \overline{F} \right)$ where
    \begin{align*}
        -\E\left[\overline{F} \Big| \mu \right] \le -\E \left[ F(X,\Lambda,\mu)\Big| \mu \right].
    \end{align*}
\end{proposition}

\begin{proof}
    Let $1 \le i \le n$ and let us introduce $\boldsymbol{S}^{i,n}:=(S^{i,j,n})_{1 \le j \le n}$ verifying: for each $1 \le j \le n$, $S^{i,j}_0=\iota^j$, 
    \begin{align*}
        \mathrm{d}S^{i,j,n}_t=b^{j,n}(t,\boldsymbol{S}^{i,j,n})\mathrm{d}t + \sigma(t,S^{i,j,n}_t) \mathrm{d}W^j_t\;\mbox{for }j \neq i,\;\mathrm{d}S^{i,i,n}_t=\sigma(t,S^{i,i,n}_t)\mathrm{d}W^i_t\;\mbox{and}\;\nu^{i,n}:=\frac{1}{n}\sum_{j=1}^n \delta_{S^{i,j,n}}.
    \end{align*}
    Also, $L^i_0=1$, $\mathrm{d}L^i_t=L^i_t \sigma(t,S^{i,i,n}_t)^{-1} b^{i,n}(t,\boldsymbol{S}^{i,n})\mathrm{d}W^i_t$ and $\mathrm{d}\P^i:=L^i_T \mathrm{d}\P$. Since $b^{i,n}$ is bounded, by Girsanov Theorem, we check that $\Lc^{\P^i}(\boldsymbol{S}^{i,n})=\Lc(\Xbb^n)$. Therefore, with $\nu^{i,n}:=\frac{1}{n} \sum_{j=1}^n \delta_{S^{i,j,n}}$, for any bounded Borel map $\Delta$
    \begin{align} \label{eq:equality_in_law}
        \E \left[ \Delta \left(X^{i,n}, \Lambda^{i,n}(\Xbb^n), \mu^n, \overline{F}^n(\Xbb^n) \right) \right]
        &=
        \E^{\P^i} \left[ \Delta \left(S^{i,i,n}, \Lambda^{i,n}(\Sbb^{i,n}), \nu^{i,n}, \overline{F}^n(\Sbb^{i,n}) \right) \right] \nonumber
        \\
        &=
        \E \left[ L^i_T \Delta \left(S^{i,i,n}, \Lambda^{i,n}(\Sbb^{i,n}), \nu^{i,n}, \overline{F}^n(\Sbb^{i,n}) \right) \right].
    \end{align}
    We set 
    $$
        \Pr^n:=\frac{1}{n} \sum_{i=1}^n \Lc\left(L^i, S^{i,i,n}, \Lambda^{i,n}(\Sbb^{i,n}), \nu^{i,n}, H(S^{i,i,n}), F^{i,n}(\Sbb^{i,n}), \overline{F}^n(\Sbb^{i,n}) \right)
    $$
    and
    $$
        \Qr^n:=\frac{1}{n} \sum_{i=1}^n \Lc\left(X^{i,n}, \Lambda^{i,n}(\Xbb^{n}),\mu^n, H(X^{i,n}),\;F^{i,n}(\Xbb^n), \overline{F}^n(\Xbb^{n}) \right).
    $$
    Since $b^{i,n}$ are bounded uniformly in $(i,n)$, we can check that $\sup_{n \ge 1} \frac{1}{n} \sum_{i=1}^n\E[|L^i_T|^q]+ \E[|(L^i_T)^{-1}|^q]<\infty$ for each $q >1$. Under the condition of the proposition, the sequences $(\Pr^n)_{n \ge 1}$ and $(\Qr^n)_{n \ge 1}$ are relatively compact for the weak convergence topology. Let $\Pr=\P \circ (L,S, {\Lambda}^{\Pr},\nu,H^{\Pr},F^{\Pr}, \overline{F}^{\Pr} )^{-1}$ and $\Qr=\P \circ (X,\Lambda^{\Qr},\mu,H^{\Qr},F^{\Qr},\overline{F}^{\Qr})^{-1}$ be the limits of a convergent sub--sequence. We use the same notation for the sequence and its sub--sequence for simplicity. If we denote $(L,S,\Lambda,\mu,H,F,\overline{F})$ the canonical process on $C([0,T];\R) \x C([0,T];\R) \x \M(\Er \x \R) \x \Pc(C([0,T];\R))\x \R^3$, for each $n \ge 1$, $\Pr^n\circ (S)^{-1}=\P \circ (U)^{-1}$ where $U_0=\iota$ and $\mathrm{d}U_t=\sigma(t,U_t)\mathrm{d}W_t$. Let $\Phi$ be continuous bounded map, by using stable convergence of \cite{jacod1981type}, 
    \begin{align*}
        &\E \left[L_T \Phi\left( S,\Lambda^{\Pr},\nu, F^{\Pr}, H^{\Pr} \right) \right]=\lim_{n \to \infty}  \E^{\Pr^n} \left[L_T \Phi\left( S,\Lambda,\nu, F, H \right) \right]
        \\
        &= \lim_{n \to \infty} \frac{1}{n} \sum_{i=1}^n  \E \left[L^i_T \Phi\left( S^{i,i,n},\Lambda^{i,n}(\Sbb^{i,n}),\nu^{i,n}, F(S^{i,i,n}, \Lambda^{i,n}(\Sbb^{i,n}),\nu^{i,n}), H(S^{i,i,n}) \right) \right]
        \\
        &= \E \left[L_T \Phi\left( S,\Lambda^{\Pr},\nu, F(S, \Lambda^{\Pr},\nu), H(S) \right) \right].
    \end{align*}
    This is true for any $\Phi$, consequently, we have ${F}^{\Pr}=F(S,{\Lambda}^{\Pr},\nu)$ and $H^{\Pr}=H(S)$ a.e.
    For any continuous bounded map $G:C([0,T];\R) \x \M(\Er \x \R) \x \Pc_p(C([0,T];\R)) \to \R$, we have
    \begin{align*}
        \lim_{n \to \infty}\frac{1}{n} \sum_{i=1}^n\E \left[ G\left(X^{i,n}, \Lambda^{i,n}(\Xbb^n),\mu^n\right) \right]
        =\lim_{n \to \infty} \frac{1}{n} \sum_{i=1}^n \E \left[ L^i_T G(S^{i,i,n}, \Lambda^{i,n}(\Sbb^{i,n}), \nu^{i,n}) \right] =\E [L_T G(S, {\Lambda}^{\Pr},\nu)]
    \end{align*}
    and $\lim_{n \to \infty}\frac{1}{n} \sum_{i=1}^n\E \left[ G(X^{i,n}, \Lambda^{i,n}(\Xbb^n),\mu^n) \right]=  \E \left[ G(X,{\Lambda},\mu) \right].$
    It is easy to verify that $L_T \ge 0$ and $\E[L_T]=1$. By setting $\mathrm{d}\overline{\P}:=L_T \mathrm{d}\P$, we have $\E^{\overline{\P}} [G(S,{\Lambda}^{\Pr},\nu)]=\E \left[ G(X,\Lambda,\mu)\right]. $
    Since the previous equality is true for any continuous bounded $G$, we deduce that $\Lc(X,\Lambda,\mu)=\Lc^{\overline{\P}}\left(S,{\Lambda}^{\Pr},\nu \right)$ and, by similar arguments to the ones did in  \eqref{eq:equality_in_law}, we can check that $\Lc^{\overline{\P}}( S,\Lambda^{\Pr}, \nu,H^{\Pr}, F^{\Pr}) = \Lc( X,\Lambda^{\Qr}, \mu,H^{\Qr}, F^{\Qr})$. Consequently
    \begin{align*}
        &\Lc( X,\Lambda^{\Qr}, \mu,H^{\Qr}, F^{\Qr})=\Lc^{\overline{\P}}( S,\Lambda^{\Pr}, \nu,H^{\Pr}, F^{\Pr})
        \\
        &= \Lc^{\overline{\P}}( S,\Lambda^{\Pr}, \nu,H(S), F(S,\Lambda^{\Pr},\nu)) = \Lc( X,\Lambda, \mu,H(X), F(X,\Lambda,\mu))
    \end{align*}
    This is true for any sub--sequences of $(\Pr^n)_{n \ge 1}$ and $(\Qr^n)_{n \ge 1}$ . Also, we can show that for any continuous bounded $v:\R \to \R$,  $\lim_{n \to \infty} \frac{1}{n} \sum_{i=1}^n \E \left[v \left(\left| H(X^{i,n}) \wedge n - H(X^{i,n}) \right| \right) \right] = v(|0|)$. To conclude, let us provide the characterization of $\overline{F}^{\Qr}$. Let us take a sequence of non--negative smooth maps $(\phi_k:\R \to \R_+)_{k \ge 1}$ s.t. for each $k$, $\phi_k$ has compact support, $\phi_k \le 1$ and $\lim_{k \to \infty}\phi_k(x)=1$ for each $x \in \R$. For a non--negative continuous map with compact support $A:\Pc_p(C([0,T];\R)) \x \R \to \R,$ for each ($i$, $n$, $k$),
    \begin{align} \label{eq:before_limit_charac}
        -\E \left[F^{i,n}(\Xbb^{n})) A(\mu^{n}, \overline{F}^n(\Xbb^{n}) ) \right] &= -\E \left[ F^{i,n}(\Xbb^{n})) \phi_k\left( F^{i,n}(\Xbb^{n}) \right) A(\mu^{n}, \overline{F}^n(\Xbb^{n}) ) \right] \nonumber
        \\
        &\;\;- \E \left[  F^{i,n}(\Xbb^{n})) \left(1-\phi_k\left( F^{i,n}(\Xbb^{n}) \right) \right)A(\mu^{n}, \overline{F}^n(\Xbb^{n}) ) \right] \nonumber
        \\
        &\le -\E \left[ F^{i,n}(\Xbb^{n})) \phi_k\left( F^{i,n}(\Xbb^{n}) \right) A(\mu^{n}, \overline{F}^n(\Xbb^{n}) ) \right].
    \end{align}
    We recall that
    \begin{align*}
        &\E \left[|F^{\Qr}| \right] \le \liminf_{n \to \infty} \frac{1}{n} \sum_{i=1}^n \E \left[|F^{i,n}(\Xbb^n)| \right] 
        \\
        &\le C \left( 1 +  \sup_{n \ge 1} \frac{1}{n} \sum_{i=1}^n \E \left[ \| \Lambda^{i,n}(\Xbb^n))\|^2_2 \right] + \E[|L^i_T|^{r'}]^{1/r'}\E[\psi(U)^r]^{1/r} \right) < \infty
    \end{align*}
    where $1/r+1/r'=1$ and $C>0$ independent of $n$. Then, in \Cref{eq:before_limit_charac}, by taking the empirical sum, the limit in $n$ and, the limit in $k$ by dominated convergence, we found 
    \begin{align*}
        -\E \left[ \overline{F}^{\Qr} A (\mu, \overline{F}^{\Qr} ) \right] \le -\lim_{k \to \infty}\E \left[ F^{\Qr} \phi_k\left( F^{\Qr} \right) A(\mu,\overline{F}^{\Qr}) \right] = -\E \left[ F^{\Qr} A(\mu,\overline{F}^{\Qr}) \right].
    \end{align*}
    This is true for any non--negative $A$, therefore
    \begin{align*}
        -\E \left[ \overline{F}^{\Qr} \Big| \mu \right] \le -\E \left[ F^{\Qr} \Big| \mu \right] = -\E \left[ F(X,\Lambda,\mu) \Big| \mu \right].
    \end{align*}
\end{proof}

\end{appendix}

\bibliographystyle{plain}

\bibliography{PA-McKeanVlasov_arxiv_version}

\begin{thebibliography}{35}
\providecommand{\natexlab}[1]{#1}
\providecommand{\url}[1]{\texttt{#1}}
\expandafter\ifx\csname urlstyle\endcsname\relax
  \providecommand{\doi}[1]{doi: #1}\else
  \providecommand{\doi}{doi: \begingroup \urlstyle{rm}\Url}\fi

\bibitem[Aurell et~al.(2022)Aurell, Carmona, Dayanikli, and
  Lauri\`{e}re]{DayanikliOptimal2022}
A.~Aurell, R.~Carmona, G.~Dayanikli, and M.~Lauri\`{e}re.
\newblock Optimal incentives to mitigate epidemics: A stackelberg mean field
  game approach.
\newblock \emph{SIAM Journal on Control and Optimization}, 60\penalty0
  (2):\penalty0 S294--S322, 2022.
\newblock \doi{10.1137/20M1377862}.
\newblock URL \url{https://doi.org/10.1137/20M1377862}.

\bibitem[Bergault et~al.(2024)Bergault, Cardaliaguet, and
  Rainer]{BergaultPA_minor_major_2024}
P.~Bergault, P.~Cardaliaguet, and C.~Rainer.
\newblock Mean field games in a stackelberg problem with an informed major
  player.
\newblock \emph{SIAM Journal on Control and Optimization}, 62\penalty0
  (3):\penalty0 1737--1765, 2024.
\newblock \doi{10.1137/23M1615188}.
\newblock URL \url{https://doi.org/10.1137/23M1615188}.

\bibitem[Briand and Hu(2008)]{briand2008quadratic}
P.~Briand and Y.~Hu.
\newblock Quadratic {BSDE}s with convex generators and unbounded terminal
  conditions.
\newblock \emph{Probability Theory and Related Fields}, 141\penalty0
  (3-4):\penalty0 543--567, 2008.

\bibitem[Carmona and Wang(2018)]{Carmona2018FiniteStateCT}
R.~Carmona and P.~Wang.
\newblock Finite--state contract theory with a principal and a field of agents.
\newblock \emph{Manag. Sci.}, 67:\penalty0 4725--4741, 2018.

\bibitem[Cvitani{\'c} et~al.(2017)Cvitani{\'c}, Possama{\"\i}, and
  Touzi]{cvitanic2014moral}
J.~Cvitani{\'c}, D.~Possama{\"\i}, and N.~Touzi.
\newblock Moral hazard in dynamic risk management.
\newblock \emph{Management Science}, 63\penalty0 (10):\penalty0 3328--3346,
  2017.

\bibitem[Cvitani{\'c} et~al.(2018)Cvitani{\'c}, Possama{\"\i}, and
  Touzi]{cvitanic2015dynamic}
J.~Cvitani{\'c}, D.~Possama{\"\i}, and N.~Touzi.
\newblock Dynamic programming approach to principal--agent problems.
\newblock \emph{Finance and Stochastics}, 22\penalty0 (1):\penalty0 1--37,
  2018.

\bibitem[Daniel(2018)]{Lacker-strong-2018}
L.~Daniel.
\newblock On a strong form of propagation of chaos for {M}ckean-{V}lasov
  equations.
\newblock \emph{Electronic Communications in Probability}, 23\penalty0 (45),
  2018.

\bibitem[Demski and Sappington(1984)]{demski1984optimal}
J.~Demski and D.~Sappington.
\newblock Optimal incentive contracts with multiple agents.
\newblock \emph{Journal of Economic Theory}, 33\penalty0 (1):\penalty0
  152--171, 1984.

\bibitem[Djete(2022)]{MFD-2020}
M.~F. Djete.
\newblock {Extended mean field control problem: a propagation of chaos result}.
\newblock \emph{Electronic Journal of Probability}, 27\penalty0
  (none):\penalty0 1 -- 53, 2022.
\newblock \doi{10.1214/21-EJP726}.
\newblock URL \url{https://doi.org/10.1214/21-EJP726}.

\bibitem[Djete(2023{\natexlab{a}})]{MFD-2020-closed}
M.~F. Djete.
\newblock Large population games with interactions through controls and common
  noise: convergence results and equivalence between open-loop and closed-loop
  controls.
\newblock \emph{ESAIM: COCV}, 29:\penalty0 39, 2023{\natexlab{a}}.
\newblock \doi{10.1051/cocv/2023005}.
\newblock URL \url{https://doi.org/10.1051/cocv/2023005}.

\bibitem[Djete(2023{\natexlab{b}})]{MFD_PA_comp}
M.~F. Djete.
\newblock Stackelberg mean field games: convergence and existence results to
  the problem of principal with multiple agents in competition.
\newblock \emph{arXiv preprint arXiv:2309.00640}, 2023{\natexlab{b}}.

\bibitem[Djete et~al.(0)Djete, Possama{\"\i}, and Tan]{djete2019general}
M.~F. Djete, D.~Possama{\"\i}, and X.~Tan.
\newblock Mckean–vlasov optimal control: Limit theory and equivalence between
  different formulations.
\newblock \emph{Mathematics of Operations Research}, 0\penalty0 (0):\penalty0
  null, 0.
\newblock \doi{10.1287/moor.2021.1232}.
\newblock URL \url{https://doi.org/10.1287/moor.2021.1232}.

\bibitem[El~Karoui and Tan(2013)]{karoui2013capacities2}
N.~El~Karoui and X.~Tan.
\newblock Capacities, measurable selection and dynamic programming part {II}:
  application in stochastic control problems.
\newblock \emph{arXiv preprint arXiv:1310.3364}, 2013.

\bibitem[Elie and Possama\"{\i}(2019)]{possamai2019contracting}
R.~Elie and D.~Possama\"{\i}.
\newblock Contracting theory with competitive interacting agents.
\newblock \emph{SIAM Journal on Control and Optimization}, 57\penalty0
  (2):\penalty0 1157--1188, 2019.
\newblock \doi{10.1137/17M1121202}.
\newblock URL \url{https://doi.org/10.1137/17M1121202}.

\bibitem[Elie et~al.(2019)Elie, Mastrolia, and
  Possama\"{\i}]{mastroliaAtale2019}
R.~Elie, T.~Mastrolia, and D.~Possama\"{\i}.
\newblock A tale of a principal and many, many agents.
\newblock \emph{Mathematics of Operations Research}, 44\penalty0 (2):\penalty0
  440--467, 2019.
\newblock \doi{10.1287/moor.2018.0931}.
\newblock URL \url{https://doi.org/10.1287/moor.2018.0931}.

\bibitem[Elie et~al.(2021)Elie, Hubert, Mastrolia, and
  Possamaï]{HubertMean2021}
R.~Elie, E.~Hubert, T.~Mastrolia, and D.~Possamaï.
\newblock Mean--field moral hazard for optimal energy demand response
  management.
\newblock \emph{Mathematical Finance}, 31\penalty0 (1):\penalty0 399--473,
  2021.
\newblock \doi{https://doi.org/10.1111/mafi.12291}.
\newblock URL \url{https://onlinelibrary.wiley.com/doi/abs/10.1111/mafi.12291}.

\bibitem[Green and Stokey(1983)]{green1983comparison}
J.~Green and N.~Stokey.
\newblock A comparison of tournaments and contracts.
\newblock \emph{The Journal of Political Economy}, 91\penalty0 (3):\penalty0
  349--364, 1983.

\bibitem[Holmstr{\"o}m(1982)]{holmstrom1982moral}
B.~Holmstr{\"o}m.
\newblock Moral hazard in teams.
\newblock \emph{The Bell Journal of Economics}, 13\penalty0 (2):\penalty0
  324--340, 1982.

\bibitem[Holmstrom and Milgrom(1987)]{Milgrom1987}
B.~Holmstrom and P.~Milgrom.
\newblock Aggregation and linearity in the provision of intertemporal
  incentives.
\newblock \emph{Econometrica}, 55\penalty0 (2):\penalty0 303--328, 1987.
\newblock ISSN 00129682, 14680262.
\newblock URL \url{http://www.jstor.org/stable/1913238}.

\bibitem[Holmstrom and Milgrom(1991)]{Holmstrom_Milgrom_multitask}
B.~Holmstrom and P.~Milgrom.
\newblock Multitask principal-agent analyses: Incentive contracts, asset
  ownership, and job design.
\newblock \emph{Journal of Law, Economics, \& Organization}, 7:\penalty0
  24--52, 1991.
\newblock ISSN 87566222, 14657341.
\newblock URL \url{http://www.jstor.org/stable/764957}.

\bibitem[Jacod and M\'emin(1981)]{jacod1981type}
J.~Jacod and J.~M\'emin.
\newblock Sur un type de convergence interm{\'e}diaire entre la convergence en
  loi et la convergence en probabilit{\'e}.
\newblock \emph{S\'eminaire de probabilit\'es de Strasbourg}, XV:\penalty0
  529--546, 1981.

\bibitem[Koo et~al.(2008)Koo, Shim, and Sung]{SungOptimal2008}
H.~K. Koo, G.~Shim, and J.~Sung.
\newblock {Optimal Multi‐-Agent Performance Measures For Team Contracts}.
\newblock \emph{Mathematical Finance}, 18\penalty0 (4):\penalty0 649--667,
  October 2008.
\newblock \doi{10.1111/j.1467-9965.2008.}

\bibitem[Laffont and Martimort(2002)]{LaffontMartimort2002}
J.-J. Laffont and D.~Martimort.
\newblock \emph{The Theory of Incentives: The Principal-Agent Model}.
\newblock Princeton University Press, 2002.
\newblock ISBN 9780691091846.
\newblock URL \url{http://www.jstor.org/stable/j.ctv7h0rwr}.

\bibitem[Laffont and Tirole(1993)]{laffont1993theory}
J.-J. Laffont and J.~Tirole.
\newblock \emph{A Theory of Incentives in Procurement and Regulation}.
\newblock MIT Press, 1993.
\newblock ISBN 9780585167978.

\bibitem[Mookherjee(1984)]{Mookherjee1984}
D.~Mookherjee.
\newblock Optimal incentive schemes with many agents.
\newblock \emph{The Review of Economic Studies}, 51\penalty0 (3):\penalty0
  433--446, 1984.
\newblock ISSN 00346527, 1467937X.
\newblock URL \url{http://www.jstor.org/stable/2297432}.

\bibitem[M{\"u}ller(1998)]{muller1998first}
H.~M{\"u}ller.
\newblock The first--best sharing rule in the continuous--time principal--agent
  problem with exponential utility.
\newblock \emph{Journal of Economic Theory}, 79\penalty0 (2):\penalty0
  276--280, 1998.
\newblock ISSN 0022-0531.
\newblock \doi{http://dx.doi.org/10.1006/jeth.1997.2381}.
\newblock URL
  \url{http://www.sciencedirect.com/science/article/pii/S0022053197923814}.

\bibitem[M{\"u}ller(2000)]{muller2000asymptotic}
H.~M{\"u}ller.
\newblock Asymptotic efficiency in dynamic principal--agent problems.
\newblock \emph{Journal of Economic Theory}, 91\penalty0 (2):\penalty0
  292--301, 2000.

\bibitem[Salani{\'e}(1997)]{salanie1997economics}
B.~Salani{\'e}.
\newblock \emph{The Economics of Contracts: A Primer}.
\newblock MIT Press, 1997.
\newblock ISBN 9780262193863.

\bibitem[Sannikov(2008)]{sannikov2008continuous}
Y.~Sannikov.
\newblock A continuous--time version of the principal--agent problem.
\newblock \emph{The Review of Economic Studies}, 75\penalty0 (3):\penalty0
  957--984, 2008.

\bibitem[Sannikov(2013)]{sannikov2012contracts}
Y.~Sannikov.
\newblock Contracts: the theory of dynamic principal--agent relationships and
  the continuous--time approach.
\newblock In D.~Acemoglu, M.~Arellano, and E.~Dekel, editors, \emph{Advances in
  economics and econometrics, 10th world congress of the Econometric Society,
  volume 1, economic theory}, number~49 in Econometric Society Monographs,
  pages 89--124. Cambridge University Press, 2013.

\bibitem[Sch{\"a}ttler and Sung(1993)]{schattler1993first}
H.~Sch{\"a}ttler and J.~Sung.
\newblock The first--order approach to the continuous--time principal--agent
  problem with exponential utility.
\newblock \emph{Journal of Economic Theory}, 61\penalty0 (2):\penalty0
  331--371, 1993.

\bibitem[Sch{\"a}ttler and Sung(1997)]{schattler1997optimal}
H.~Sch{\"a}ttler and J.~Sung.
\newblock On optimal sharing rules in discrete--and continuous--time
  principal--agent problems with exponential utility.
\newblock \emph{Journal of Economic Dynamics and Control}, 21\penalty0
  (2):\penalty0 551--574, 1997.

\bibitem[Schmitz(2006)]{SchmitzReview}
P.~W. Schmitz.
\newblock \emph{Journal of Institutional and Theoretical Economics (JITE) /
  Zeitschrift für die gesamte Staatswissenschaft}, 162\penalty0 (3):\penalty0
  535--540, 2006.
\newblock ISSN 09324569.
\newblock URL \url{http://www.jstor.org/stable/40752600}.

\bibitem[Sung(1995)]{sung1995linearity}
J.~Sung.
\newblock Linearity with project selection and controllable diffusion rate in
  continuous--time principal--agent problems.
\newblock \emph{The RAND Journal of Economics}, 26\penalty0 (4):\penalty0
  720--743, 1995.

\bibitem[Villani(2008)]{villani2008optimal}
C.~Villani.
\newblock \emph{Optimal transport: old and new}, volume 338 of
  \emph{Grundlehren der Mathematischen Wissenschafte}.
\newblock Springer, 2008.

\end{thebibliography}




\end{document}